\documentclass[11pt]{article}
\usepackage[a4paper, margin=1in]{geometry}
\usepackage{amssymb}

\usepackage[noend, ruled, linesnumbered, noline, longend]{algorithm2e}

\usepackage{natbib}

\usepackage{mathrsfs}
\usepackage{fancyhdr}
\usepackage{tikz}
\usetikzlibrary{calc}
\newcommand{\textobullet}{\textopenbullet}
\usepackage{pgfplots}
\usetikzlibrary{arrows.meta}
\usepackage{booktabs}
\usepackage[]{mathtools}
\usetikzlibrary{backgrounds}
\usepackage{bbm}
\usetikzlibrary{positioning}
\usetikzlibrary{patterns}
\usetikzlibrary{shapes.geometric}
\usetikzlibrary{decorations.text}
\usepackage{amsthm}
\usepackage{csquotes}

\newcommand{\patheq}{\rightarrow}
\newcommand{\leqc}{\mathrel{ \leq_C} }
\newcommand{\lc}{\mathrel{<_C}}

\DeclareMathOperator{\Exists}{ \exists}
\DeclareMathOperator{\Forall}{\forall}
\DeclareMathOperator{\Nexists}{\nexists}

\newcommand{\N}{\mathbb{N}}
\newcommand{\R}{\mathbb{R}}
\newcommand{\automaton}{(Q,\Sigma,\delta,q_0,\mathcal{F})}
\newcommand{\POEAP}{POEAP}
\newcommand{\POEAPN}{$\text{POEAP}_0$}
\newcommand{\POEAPND}{$\text{POEAP}_0^{DEC}$}
\usepackage{textcomp}
\usepackage{hyperref}
\usepackage[caption=false]{subfig}

\newcommand{\stpaths}{P_{s,t}}
\newcommand{\pprice}{P^*_{s,t}}
\newcommand{\pstate}{P^f_{s,t}}
\DeclareMathOperator{\reach}{R}

\newcommand{\fns}{\footnotesize}

\newcommand{\monoidel}{h}
\newcommand{\monoidset}{H}
\newcommand{\neutr}{\boldsymbol{0}}

\newcommand{\pmonoid}{(\monoidset,+,\leq)}
\newcommand{\zonemonoid}{H_{zone}}
\newcommand{\citymonoid}{H_{c}}
\newcommand{\transfermonoid}{H_{tran}}
\newcommand{\stopmonoid}{H_{stop}}
\newcommand{\distmonoid}{H_{dist}}
\newcommand{\city}{c}
\newcommand{\cities}{C}
\newcommand{\fareevents}{S}
\newcommand{\fareevent}{s}
\newcommand{\nullevent}{s_0}
\newcommand{\ticketgraph}{\mathcal{T}}
\newcommand{\tickets}{T}
\newcommand{\ticket}{\tau}
\newcommand{\tarcs}{E}
\newcommand{\tarc}{e}
\newcommand{\rarc}{a}
\newcommand{\rarcs}{A}
\newcommand{\attributespace}{W}

\newcommand{\statespace}{F}
\newcommand{\state}{f}
\newcommand{\altstate}{g}
\newcommand{\startstate}{\mu}
\newcommand{\transfunc}{\Gamma}
\newcommand{\update}{\mathrm{Up}}

\newcommand{\price}{\pi}
\newcommand{\farenetwork}{\mathcal{N}}

\newcommand{\tf}{\tau}
\newcommand{\ef}{e}
\newcommand{\wfs}{w}  
\newcommand{\wfa}{w}  
\newcommand{\sixtuple}{( \ticketgraph,\transfunc,\wfa,\ef,\startstate,\price) }

\newcommand{\mdvz}[1]{Z_{#1}}
\newcommand{\mdvmax}{M}
\newcommand{\mdvdis}{D}
\newcommand{\mdvkl}{D_L}
\newcommand{\mdvkh}{D_H}
\newcommand{\mdvl}{L}
\newcommand{\mdvh}{H}
\newcommand{\mdvt}[1]{C_{#1}}

\newcommand{\eventhalle}{\mathit{hal}}
\newcommand{\eventleipzig}{\mathit{lei}}
\newcommand{\eventtrans}{\mathit{tra}}
\newcommand{\eventcity}{\mathit{city}}

\newcommand{\RAP}{RAPTOR}
\newcommand{\MCRAP}{McRAP}
\newcommand{\BMRAP}{BMRAP}
\newcommand{\TightBMRAP}{Tight-BMRAP}
\newcommand{\TargetBMRAP}{Target-BMRAP}
\newcommand{\TTPI}{\Pi}
\newcommand{\TTP}{\mathcal{P}}
\newcommand{\TTR}{\mathcal{R}}
\newcommand{\TTT}{\mathcal{D}}
\newcommand{\TTF}{\mathcal{F}}
\newcommand{\TT}{\mathbbm{T}}
\newcommand{\TimeTable}{\TT = (\TTPI, \TTP, \TTR,\TTT,\TTF)}
\DeclareMathOperator{\tr}{tr}
\newcommand{\raptime}{\eta}
\newcommand{\arr}{\raptime_{arr}}
\newcommand{\dep}{\raptime_{dep}}
\newcommand{\asig}{\sigma_{arr}}
\newcommand{\tsig}{\sigma_{tr}}

\newcommand{\tripweight}{\wfa_1}
\newcommand{\tripevent}{\ef_1}
\newcommand{\transweight}{\wfa_2}
\newcommand{\transevent}{\ef_2}

\newcommand{\janch}{\mathcal{J}_{\mathcal{A}}}
\newcommand{\jr}{\mathcal{J}_{\mathcal{R}}}
\newcommand{\jstate}{\mathcal{J}^f}
\newcommand{\jprice}{\mathcal{J}^*}
\newcommand{\jrstate}{\mathcal{J}^f_{\mathcal{R}}}
\newcommand{\jrprice}{\mathcal{J}^*_{\mathcal{R}}}
\newcommand{\sjanch}{{J}_{\mathcal{A}}}
\newcommand{\jdelling}{\bar{ \mathcal{J}_{\mathcal{R}}}}
\newcommand{\backdep}{\raptime_{dep}}
\newcommand{\transtime}{\raptime_{ch}}
\newcommand{\raptrip}{d}
\newcommand{\fssset}{\mathcal{J}^{fss}}
\newcommand{\rfssset}{\mathcal{J}^{fss}_{\mathcal{R}}}
\newcommand{\station}[1]{\mathcal{#1}}

\newcommand{\fullcomp}{C_F}
\newcommand{\partcomp}{C_P}
\newcommand{\nocomp}{C_N}

\pgfplotsset{compat=1.16}

\providecommand{\keywords}[1]{\noindent\textbf{\textit{Index terms---}} #1}
\providecommand{\competinginterest}[1]{\noindent\textbf{\textit{Declaration of interest---}} #1}

\newtheorem{definition}{Definition}[section]
\newtheorem{proposition}{Proposition}[section]
\newtheorem{lemma}{Lemma}[section]
\newtheorem{theorem}{Theorem}[section]

\theoremstyle{definition}
\newtheorem{example}{Example}[section]

\tikzset{ farenodeStyle/.style={ draw=black,thick ,circle,minimum size=1cm,font=\large } }
\tikzset{ farearcStyle/.style={thick,-{Latex[length=2mm, width=2mm]}} }
\newlength{\mytextl}

\newcommand{\farenode}[4]
{ \node[farenodeStyle, #4] (#1) at #2 {#3}; }
\newcommand{\farearc}[5]
{
	\draw[farearcStyle] (#1) edge (#2); 
    \settowidth{\mytextl}{ \pgfinterruptpicture  #3 \endpgfinterruptpicture}
     \message{The text height is \the\mytextl} 
	\ifthenelse{#5 = 1}
	{
	  \ifthenelse{\lengthtest{ #4 pt > 0.99pt}}
	  {
	    \path [postaction={decorate,decoration={raise=2mm,text along path,text={#3{}},text align={left, left indent=0cm}}}] (#2) to [] (#1); 
	  }
	  {
	    \ifthenelse{\lengthtest{ #4 pt < 0.01pt}}
	    {
	      \path [postaction={decorate,decoration={raise=2mm,text along path,text={#3{}},text align={right, right indent=0cm}}}]      (#2) to [] (#1); 
	    }
	    {
	      \path [postaction={decorate,decoration={raise=2mm,text along path,text={#3{}}, text align={left indent={#4\dimexpr\pgfdecoratedpathlength\relax}} }}]      (#2) to [] (#1); 
	    }
	  }
	}
	{
	  \ifthenelse{\lengthtest{ #4 pt < 0.01pt}}
	  {
	    \path [postaction={decorate,decoration={raise=2mm,text along path,text={#3{}},text align={left, left indent=0cm}}}]      (#1) to [] (#2); 
	  }
	  {
	    \ifthenelse{\lengthtest{ #4 pt > 0.99pt}}
	    {
	      \path [postaction={decorate,decoration={raise=2mm,text along path,text align=center,text={#3{}},text align={right, right indent=0cm}}}]  (#1) to [] (#2); 
	    }
	    {
	      \path [postaction={decorate,decoration={raise=2mm,text along path,text align={left indent={#4\dimexpr\pgfdecoratedpathlength\relax}},text={#3{}} }}]      (#1) to [] (#2); 
	    }
	  }
	}	
}  

\hypersetup{
	pdftitle={Price Optimal Routing in Public Transportation},
	pdfsubject={},
	pdfauthor={Ricardo Euler, Niels Lindner, Ralf Borndörfer},
        pdfkeywords={multi-objective, shortest path, MOSP, fare, fare structure, public transportation, monoid, raptor,tariff, ticket, ticket graph, earliest arrival problem, RAPTOR},
}

\title{Price Optimal Routing in Public Transportation}
\date{}
\author{ \href{https://orcid.org/0000-0001-5112-4191}{\includegraphics[scale=0.06]{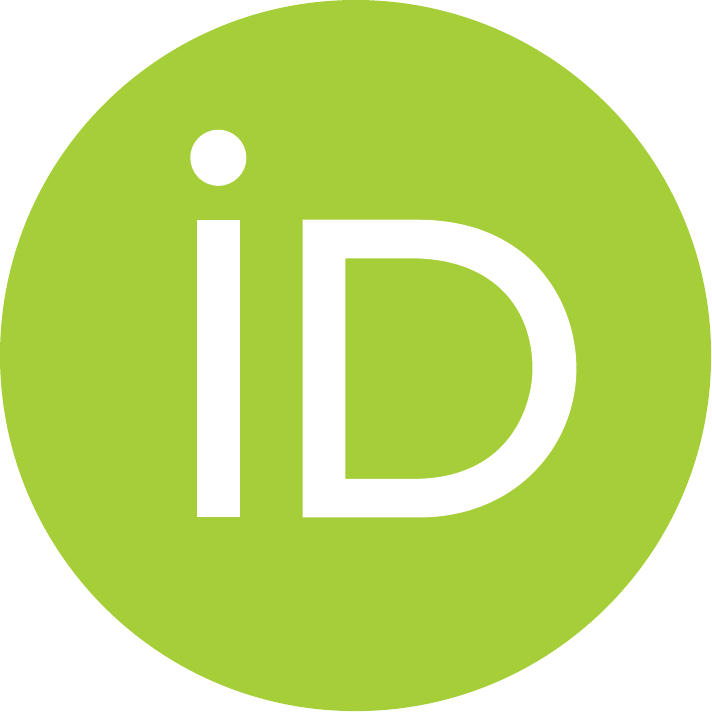}\hspace{1mm}Ricardo Euler
    \footnote{The authors work was partly conducted within the Research Campus MODAL funded by the German Federal Ministry of Education and Research (BMBF) (fund number 05M20ZBM) and partly funded by  the Federal Ministry of Transport and Digital Infrastructure (BMVI) under the project no. 19E17001C. The funding sources were not involved in any form in the preparation of this manuscript.   }
    }
    \href{https://orcid.org/0000-0002-8337-4387 }{\includegraphics[scale=0.06]{orcid_id_icon.pdf}\hspace{1mm}Niels Lindner}
        \href{https://orcid.org/0000-0001-7223-9174}{\includegraphics[scale=0.06]{orcid_id_icon.pdf}\hspace{1mm}Ralf Borndörfer}\\
	Network Optimization \\
	Zuse Institute Berlin\\
	Germany, Berlin, 14195\\
	\texttt{\{euler,lindner,borndoerfer\}@zib.de}\\ }

\usepackage{fancyhdr}
\begin{document}

\maketitle

\begin{abstract}
We consider the \textit{price-optimal earliest arrival problem} in public transit (\POEAP) in which we aim to calculate the Pareto-set of journeys with respect to ticket price and arrival time in a public transportation network.
Public transit fare structures are often a combination of various fare strategies such as, e.g., distance-based fares,  zone-based fares or flat fares.
The rules that determine the actual ticket price are often very complex.
Accordingly, fare structures are notoriously difficult to model, as it is in general not sufficient to simply assign costs to arcs in a routing graph.
Research into \POEAP{} is scarce and usually either relies on heuristics or only considers restrictive fare models 
that are too limited to cover the full scope of most real-world applications. We therefore introduce \textit{conditional fare networks} (CFNs), the first framework for representing a large number of real-world fare structures.
We show that by relaxing label domination criteria, CFNs can be used as a building block in label-setting multi-objective shortest path algorithms.
By the nature of their extensive modeling capabilities, optimizing over CFNs is NP-hard.
However, we demonstrate that adapting the  multi-criteria RAPTOR (\MCRAP) algorithm for CFNs yields an  
algorithm capable of solving \POEAP{} to optimality in less than 400 ms on average
on a real-world data set.
By restricting the size of the Pareto-set, running times are further reduced to below 10 ms.
\end{abstract}


\competinginterest{none}

\section{Introduction}
The desired shift to more sustainable means of transportation necessitates an increase in the modal share of public transportation systems.
Recent studies show that discounted or even free fares significantly increase such a system's adaptation \citep{Brough2022,Bull2021,Chen}, indicating that, from a traveler's perspective, the design of fares and ticket prices are key parameters to foster its attractiveness. 
This effect, however, varies between socioeconomic groups: While many riders may value fast connections with few transfers, low-income riders are more likely to choose a more affordable mode of public transportation or to even abstain from using transportation at all if they deem fares to be too expensive \citep{Blumenberg2014, Rosenblum2020}.
This underlines the need for routing algorithms that enable passengers to select journeys optimal with respect to their individual needs.
Since these needs are rarely known, we present an approach that calculates Pareto sets of optimal journeys with respect to the earliest arrival time, the number of transfers and cost.
It is then on the user to choose from the options presented the one most suited for their personal needs.

Ticket prices are determined by the public transit providers' fare structure.
Following \citet{fleishman1996fare} a \textit{fare structure} is \enquote{the combination of one or more
fare strategies with specific tickets} while the term
\textit{fare strategy} refers to a \enquote{general fare collection and payment structure approach}.
This can be, e.g, a distance-based  or zone-based fare, a short-distance discount, a flat fare or a surcharge.
Until now, a unifying framework for the algorithmic treatment of public transit fare structures has been lacking.
Some fare strategies, for example distance-based fares, can be addressed easily using label-setting shortest path algorithms.
Finding a journey that crosses the least amount of fare zones, however, is NP-hard \citep{BlancoBorndoerferHoangetal.2016} and can not be modeled using real-valued arc weights.
In general, the subpath optimality principle does not hold for fare structures and thus label-setting algorithms can not be directly applied to \POEAP.
Consider the following example: A traveler takes a detour that passes through an additional fare zone to avoid paying the surcharge of a special connection (e.g., a ferry). 
It is not unlikely that, at a later point, the surcharge has to be paid regardless (e.g., because the target stop can only be reached via a ferry).
In that case, taking the detour was a suboptimal decision and using the corresponding label to prune other partial journeys breaks the optimality guarantee of shortest path algorithms.

\subsection{Our Contribution}
We devise algorithms that solve the price-optimal earliest arrival problem (\POEAP{}) efficiently in practice.
To this end, we build upon the state-of-the-art public transport routing algorithms \MCRAP\ \citep{DBLP:journals/transci/DellingPW15} and \TightBMRAP\ \citep{doi:10.1137/1.9781611975499.5}, which combine an intelligent enumeration scheme with dominance checks.
Our domination rules are based on \textit{conditional fare networks} (CFN), a novel and flexible framework for modeling fare structures of public transportation providers capable of taking most underlying fare strategies into account.
A CFN models fare structures as a \textit{ticket graph} representing \textit{relations} between tickets.
Transitions between different tickets are modeled as directed arcs and usually depend on a number of additional parameters such as, e.g., fare zones or the traveled distance. These are modeled via partially ordered monoids and events.
Fare strategies that can be expressed via CFNs include (but are not limited to): zone-based fares, distance-based fares, surcharges for special vehicles or night liners, discounted short-distance fares, 
transfer fares and all combinations thereof. 
We develop domination rules for CFNs based on path relations in the ticket graph instead of the price alone. This  allows us to retain 
subpath optimality and prove that using these rules in label-setting MOSP algorithm does in fact yield lowest-price journeys.
By further restricting the size of the Pareto-set, we solve \POEAP{} in less than ten milliseconds over the intricate fare structure of a mid-sized public transit provider from Germany.

\subsection{Related Literature}\label{sec:rellit}

There is ample research on routing problems in public transit networks. For an overview, see \citet{Bast2016RoutePI}.
\citet{DBLP:journals/transci/DellingPW15} introduced the \RAP{} algorithm for very fast public transit routing.
We refer occasionally to typical label-setting multi-objective shortest path algorithms,
by which we mean, e.g., 
\MCRAP\ \citep{DBLP:journals/transci/DellingPW15}, Martins' algorithm \citep{martins1984multicriteria}, and recently
Multi-Objective Dijkstra \citep{MaristanydelasCasasSedenoNodaBorndoerfer2021}.
\citet{doi:10.1137/1.9781611975499.5} introduce a version of \MCRAP, \TightBMRAP , for computing restricted Pareto-sets. 

The concept of relaxed subpath optimality is discussed in \citet{Berger_subpath-optimalityof}.
Lastly, we refer to \citet{Disser2008MulticriteriaSP} on how to model public transit systems with transfers in time-dependent graphs.

In contrast to the general activity of the field, literature on price-optimal routing is generally rather scarce.
This is certainly due to the usually intricate nature of public transit fare structures.
Most approaches deal with fares on a heuristic basis or only consider a very narrow set of fare strategies.
Most notably, \citet{HannemannSchneePayingLess} study fare structures that entail distance- and relation-based prices, i.e.,
structures that are usually associated with long-distance public transportation.
They approximate fares by assigning a fixed price to every arc. 
This approach, however, does not account for fare strategies such as, e.g.,  fare zones and short-distance discount tickets.
Both are usually more prominent in local public transportation.

\citet{Reinhardt2011} consider optimizing the number of fare zones as a special case in their study of non-additive objective functions in (multi-criteria) shortest path problems.
However, their approach relies on target pruning as the sole domination technique, so that partial paths cannot be pruned until a $s$,$t$-path is known. In contrast, our approach applies to more general fare structures, includes target pruning, but also allows for pruning at earlier stages.

\citet{schoebel2021cheapest} identified conditions under which price-optimized routing is tractable for zone- and distance-based fare strategies.
Additionally, they identify the no-elongation and no-stopover properties as desirable properties for fare structures.
This follows a line of research concerned with the design of fare structures. For a review of recent work, see \citet{schoebel2021cheapest}.
\citet{BlancoBorndoerferHoangetal.2016} showed zone-based fares to result in NP-hard routing problems if reentering a zone does not entail additional costs.
The proof relies on a reduction of the problem of finding a path with a  minimum number of colored edges \citep{broersma2005paths}.
It was given in the context of flight trajectory optimization with overflight costs \cite[]{BlancoBorndoerferHoangetal.2017}.
A version adapted for public transport is given by \citet{schoebel2021cheapest}.
\citet{DBLP:journals/transci/DellingPW15} used \RAP{} to compute journeys that touch the smallest number of fare zones.
Recently, \citet{GundlingDissertation2020} considered price-optimized routing in intermodal transportation. They, however, only considered mileage-based and flat fares.

Our approach combines ideas from automata theory and optimization over monoids.
For an overview of automata theory, see \citet{Hopcroft+Ullman/79/Introduction}.
The ticket graph concept is inspired  by the application of finite automata to the language-constrained shortest path problem \citep{Barrett:2000:FPP:586846.586970}. It is different, however, in that it serves to evaluate paths instead of restricting the set of feasible paths. Furthermore, our approach also covers fares based on numerical parameters that are not expressed as part of a formal language.
Finally, it has been known for a while that shortest path algorithms can be generalized to ordered monoids \citep{ZimmermannOrderedCombinatorialOptimization}  and  semirings \citep{Mohri:2002:SFA:639508.639512} in a straightforward fashion. Recently, monoids were also proposed as a general constraint model for (single-criteria) resource constrained shortest path problems \citep{Parmentier2019a}.

This paper is an extended and improved version of work presented at the ATMOS'19 conference \citep{euler_et_al:OASIcs:2019:11424}. 
Apart from streamlining the presentation and proofs, the following additions were made:
We now better motivate the interplay between the monoid and the ticket graph.
We propose a superior approach for dealing with overlap areas.
We provide an adaption of the recent Tight-BMRAP algorithm to our use case.
This leads to an improvement in algorithmic performance of up to two orders of magnitude compared to the previous results in \citet{euler_et_al:OASIcs:2019:11424}.
Finally, we provide a complexity analysis and investigate the relation of our approach to automata theory
in \ref{sec:complexity}.

Ticket graphs of various German public transit providers can be found in \citet{BorndoerferEulerKarbsteinetal.2018} and \citet{BorndoerferEulerKarbstein2021}.

\subsection{Overview}
In Section \ref{section:MDV}, we introduce the fare structure of MDV, an association that is responsible for the public transit fares for various operators in the Leipzig-Halle region of Germany. This fare structure will serve as a running example for the rest of the paper.
We present conditional fare networks in detail in Section \ref{sec:framework} and show how they can be used to model various aspects of fare structures.
The algorithmic treatment of fares and domination rules is laid out in Section \ref{sec:algo}. 
Section \ref{sec:psrap} discusses how the multi-criteria \RAP\ and \TightBMRAP{} algorithms can be modified to use CFNs for price-optimal search.
An evaluation of the framework's performance is conducted in Section \ref{sec:results} using the network and fare structure of MDV.
Section \ref{sec:conclusion} concludes the paper with some closing remarks. 
In \ref{sec:complexity}, we provide a supplementary complexity analysis of \POEAP{} and explore links to automata theory.

\section{Running Example: MDV}\label{section:MDV}

We introduce the reader to some intricacies of fare structures in public transit using the example of \textit{Mitteldeutscher Verkehrsverbund} (MDV) \cite{MDVFare2019}.
Throughout Sections \ref{sec:framework} and \ref{sec:algo}, the MDV fare structure will serve as a running example to illustrate our core concepts.  
A schematic depiction of MDV's fare plan is given in Figure \ref{fig:mdv_plan}.

\begin{example}[The Fare System of MDV]\label{example:mdv_fares}
MDV's  area of operations covers large rural areas in Eastern Germany, as well as the conurbation of Halle and Leipzig. As of 2019, this area is divided into a set of 56 pairwise disjoint fare zones.
In most cases, the price depends on the number of visited fare zones: there are price levels for one to six fare zones.
We denote the respective tickets by $\mdvz{i}$ with $i\in [6]$. 
For example, traveling from station $\station{A}$ to station $\station{L}$ in Figure \ref{fig:mdv_plan} requires ticket $\mdvz4$.
For all paths covering more than six fare zones, a ticket for MDV's whole area of operations has to be purchased, which we denote by $M$.
The two larger cities, Halle and Leipzig, each form a single fare zone. Travelling in these zones requires special tickets more expensive than $\mdvz1$.
These, we denote by $\mdvh$ and $\mdvl$, respectively.

\begin{figure}[t]
    \centering
    \begin{tikzpicture}
    \footnotesize
    \tikzset{
        zonebox/.style={thick},
        city/.style={fill=gray!30},
        neutralcleaner/.style = {thick,regular polygon, regular polygon sides=6,draw, fill=white,minimum size=1.2cm},
        neutral/.style={ thick,regular polygon, regular polygon sides=6,draw, fill=lightgray, minimum size=1.2cm, pattern=north west lines, pattern color=lightgray},
        mainstation/.style={},
        town/.style={circle, minimum size=1.7cm ,draw, thick, pattern=horizontal lines,pattern color=lightgray},  
        station/.style={draw,circle, minimum size=0.133cm, fill=black, inner sep=0},
        line/.style={draw, line width=0.04cm}}


    \draw[zonebox,city] (0,0) -- (0,1.33)
            -- (.66,2)
            -- (2,2)
            -- (2.66,1.33) 
            -- (4.33,1.33) 
            -- (4.33,.66) coordinate (contact_top_1)  
            -- (2.66,.66) coordinate (contact_top_2)  
            -- (2.66,0) coordinate (contact_top_3)  
            -- (1.33,-1.33) coordinate (contact_top_4)  coordinate (neutral_1)
            -- (0,0)
            -- (0,1.33);
    \draw[zonebox]
            let \p1 = (contact_top_1) in
              (contact_top_4)
              -- (4.33,-1.33)
              -- (4.33,-2.66)
              -- (6,-2.66)  coordinate (contact_bottom_1)
              -- (6,-0.66)  coordinate (neutral_2)
              -- (6,\y1) coordinate (contact_top_six)
              -- (\p1) ;
    \draw[zonebox] let \p1 = (contact_bottom_1), \p2 = (contact_top_six) in
            (\p2) 
            -- (9.33,\y2) coordinate (neutral_3)
            -- (9.33,-2) coordinate (city_two_corner_1)
            -- (8.66,-2.66) coordinate (city_two_corner_2)
            -- (\p1);
    \draw[zonebox,city] let \p1 = (city_two_corner_1), \p2 = (city_two_corner_2)  in
               (\p2)
            -- (\p1)
            -- ($ (\p1) + (2.66,0)$)
            -- (12,-5.33)
            -- (\x2,-5.33 ) coordinate (contact_bottom_2)
            -- (\p2);
    \draw[zonebox] let \p1 =(city_two_corner_2), \p2 = (contact_bottom_2) in
                (\p2)
            -- (4,\y2) coordinate (contact_bottom_3)
            -- (4,\y1) coordinate (contact_bottom_4)
            -- ( $0.3*(4,\y1) + 0.7*(\p1)$ ) coordinate (neutral_4)
            -- (\p1);
    \draw[zonebox] let \p1 = (contact_bottom_3), \p2 = (contact_bottom_4) in
            (\p2)
            -- (\p1)
            -- (-1.33,\y1)
            -- (-1.33,0)
            -- (0,0);

    \foreach \i in {1,...,4}
    {
    \node[neutralcleaner] at (neutral_\i) {};
    \node[neutral] at (neutral_\i) {};
    }

    \node[town] (merseburg) at (1.33,-4) {};

    \draw[line]
       (1.33,1.33)     node[station] (L1_1)  {}
    -- (1.33,0.833)  node[station] (L1_2)  {}
    -- (1.33,0.33)   node[station] (L1_3)  {}
    -- (1.33,-0.166) node[station] (L1_4)  {}
    -- (1.33,-0.6)    node[station] (L1_5)  {}
    -- (1.33,-1.33)    node[station] (L1_6)  {}
    -- (1.33,-3.533)  node[station] (L1_7)  {}
    -- (1.33,-4.06)  node[station] (L1_8)  {}
    -- (1.33,-4.466)  node[station] (L1_9)  {}
    -- (3.33, -4.466) node[station] (L1_10)  {}
    -- (6.66,-4.466) node[station] (L1_11)  {}
    -- (11.33,-4.466) node[station] (L1_12)  {};

    \draw[line] let \p1 = (neutral_4),\p2 = (neutral_2), \p3 = (neutral_3) in
       (11.33,-2.33) node[station]  (L2_1)  {}
    -- (11.33,-2.833) node[station] (L2_2)  {}
    -- (11.33,-3.33) node[station]    (L2_3)  {}
    -- (8,-3.33) node[station]    (L2_4)  {}
    -- (\p1)     node[station]  (L2_5)  {}
    -- (\p2)     node[station]  (L2_6)  {}
    -- ( $0.5*(\p2) + 0.5*(\p3)$ ) node[station] (L2_7)  {}
    -- (\p3)     node[station]  (L2_8)  {};
    \draw[line,dotted] (L2_3) -- (L1_12);

    \node[left  = of L1_1,  xshift=1cm,  font=\bfseries]  {$\mathcal{ A }$};
    \node[left  = of L1_3,  xshift=1cm,  font=\bfseries]  {$\mathcal{ B }$};
    \node[left  = of L1_5,  xshift=1cm,  font=\bfseries]  {$\mathcal{ C }$};
    \node[left  = of L1_6,  xshift=1cm,  font=\bfseries]  {$\mathcal{ D }$};
    \node[left  = of L1_7,  xshift=1cm,  font=\bfseries]  {$\mathcal{ E }$};
    \node[left  = of L1_9,  xshift=1cm,  font=\bfseries]  {$\mathcal{ F }$};
    \node[above = of L1_10, yshift=-1cm, font=\bfseries]  {$\mathcal{ G }$};
    \node[below = of L1_12, yshift=1cm,  font=\bfseries]  {$\mathcal{ H }$};
    \node[right = of L2_3,  xshift=-1cm, font=\bfseries]  {$\mathcal{ I }$};
    \node[right = of L2_5,  xshift=-1cm, font=\bfseries]  {$\mathcal{ J }$};
    \node[left  = of L2_6,  xshift=1cm,  font=\bfseries]  {$\mathcal{ K }$};
    \node[right = of L2_1,  xshift=-1cm, font=\bfseries]  {$\mathcal{ L }$};
    \node[above = of L2_8,  yshift=-1cm, font=\bfseries]  {$\mathcal{ M }$};
        
    \coordinate (lb_hall) at (.66,2);
    \path let \p1 = (city_two_corner_1)  in
        ($(\p1) + (0,0)$) coordinate (lb_leip);
    \node[below right = of lb_hall, yshift=1cm, xshift=-1cm, font=\bfseries] {Halle};
    \node[below right = of lb_leip, yshift=1cm, xshift=-1cm, font=\bfseries] {Leipzig};
    \node[below = of merseburg, yshift=1cm, font=\bfseries] {Merseburg};
    
    \coordinate (one_zone) at (-1.33,0);
    \node[below right = of one_zone, yshift=1cm, xshift=-1cm, font=\bfseries]  {233};
    \node[below right = of contact_bottom_4, yshift=1cm, xshift=-1cm, font=\bfseries] {156};
    \node[below right = of contact_top_2, yshift=1cm, xshift=-1cm, font=\bfseries] {225};
    \node[below right = of contact_top_six, yshift=1cm, xshift=-1cm, font=\bfseries] {162};
    \end{tikzpicture}    
    \caption{A section of MDVs fare plan with two lines and six fare zones.
    Two of the fare zones, colored in light gray, are the cities of Halle and Leipzig.
    Vertically hatched hexagons represent overlap areas that can be counted as either of the neighboring zones. 
    The horizontally hatched circle represents the small city Merseburg in which a special discounted fare is applicable.
    Small black nodes represent public transit stops. Footpaths are indicated by dotted lines.
    }
    \label{fig:mdv_plan}
\end{figure}

For all paths that pass through multiple fare zones, they, however, count as normal zones, i.e., one of the tickets $\mdvz2, \dots, \mdvz6, \mdvmax$ is applied.
Hence, the paths $\station{A}-\station{C}$ and  $\station{H}-\station{L}$ incur tickets $\mdvh$ and $\mdvl$, respectively, while  the path $\station{A}-\station{G}$ incurs $\mdvz2$.
Several smaller cities are part of larger fare zones, but allow for discounted fares (city fares) when traveling only in that city.
For each city $\city$  in the set of such cities $\cities$, we denote the ticket by $\mdvt{\city}$. 
The path $\station{E}-\station{F}$ in Merseburg ($m$), hence, requires the ticket $\mdvt{m}$. When extending the path to stop $\station{G}$, the ticket $\mdvz1$ becomes applicable.
As of 2019, there are 17 cities with city fares and two price levels (which we denote by $\mdvt1$ and $\mdvt2$).
For paths starting in Halle and Leipzig, there are discounted tickets for short trips ($\mdvkh$ and $\mdvkl$), which can be used for a maximum number of four stops without transfers.
Hence, paths $\station{A}-\station{B}$ and $\station{I}-\station{L}$ are admissible for discounted tickets $\mdvkh$ and $\mdvkl$, respectively, while paths $\station{A}-\station{C}$ and $\station{H}-\station{L}$ are not.
Discounted tickets also exist for other zones ($\mdvdis$). 
These are a little cheaper and depend on the length of the journey (4 km maximum) instead of the number of visited stops. 
Sometimes it is possible to choose between city fares and length-based discounts. In this case, the city fare is applied because it is cheaper.
To not unduly burden people living at the borders of fare zones, MDV uses overlap areas. 
These can be counted as part of either of their adjacent fare zones, whichever is most benevolent to the traveler.
For example, when traveling from $\station{J}$ to $\station{M}$, all stops are counted as part of fare zone $162$ and thus ticket $\mdvz1$ would be applicable.
When traveling from $\station{E}$ to $\station{D}$, $\station{D}$ counts as part of the fare zone  $233$ but in the path $\station{A}-\station{D}$ it counts as part of Halle.
Hence, tickets $\mdvz1$ and $\mdvh$ are applicable, respectively.

While we cover the most important features of the fare structure, we do ignore some edge cases and explicit exceptions. 
These are among other things: Slightly different discount rules for specific trains, counting stations that are passed without a stop for discounted tickets and exceptions for a specific tunnel.
This is done in part because they are not properly reflected in our data set, and in part to simplify presentation.
\end{example}

\section{A Formal Framework for Fare Structures}\label{sec:framework}
Consider a (directed) routing graph $G=(V,\rarcs)$, in which arcs represent either public transport connections, footpaths, or transfers between lines and/or modes
of transportation. Public transit journeys can then be interpreted as paths in $G$.
In the following, we will consider a time-dependent formulation as presented, for example, by \citet{Disser2008MulticriteriaSP},
i.e., we are given a time-dependent FIFO travel time function $c(a): I \rightarrow I$ on each arc $a \in A$, where $I$ is the set of time points.

In $G$, every path $p$ is associated with a ticket $\ticket \in \tickets$ from a ticket set $\tickets$ that has to be bought to use $p$. Each
ticket has a corresponding price $\price(\ticket) \in \mathbb{Q}^+$.
In the following, we might also write $\price(p)$ instead of $\price(\ticket)$ if $\ticket$ is the ticket associated with $p$.
The ticket of a path is determined by the fare structure.

We aim to solve the \textit{price-optimal earliest arrival problem} (\POEAP).
We refer to Definition~\ref{def:poeap} for a precise definition, but the essence is that, for given $s,t\in V$, we want to find a Pareto-set of $s,t$-paths $\pprice \subseteq \stpaths$  with respect to arrival time and ticket price in $G$. Here, $\stpaths$ denotes the set of all $s,t$-paths in $G$.

Ideally, we want to solve \POEAP{} by taking advantage of the existing literature on label-setting MOSP algorithms.
Ticket prices, however, usually cannot be modeled via real-valued FIFO functions on arcs.
Hence, a framework for fare structures is needed that allows us to label paths in a way that a) 
labels can be updated quickly when a new arc is relaxed, 
b) dominance relationships between labels can be established that respect the subpath optimality property.

\subsection{Modeling with Monoids}\label{subsec:modelmonoids}

Note that in Example \ref{example:mdv_fares}, the price of a path depends on several  parameters: 
the number of visited stations, the total distance traveled, the set of visited fare zones 
and on indicators reporting whether transfers were made or whether the path crossed city borders.  

All these parameters share several key properties:
First, there is a natural partial order on them, indicating which configuration requires a more expensive ticket.
For example, for fare zones $A,B,C$ we have  $\{A,B \} \subset \{A,B,C \}$; for distances and transfers it is the canonical order on $\N$.
The parameters can be summed up along a path using an appropriate notion of addition.
For distances, this is the normal addition of natural numbers; for fare zones, it is the union of sets; we can use the logical \textit{OR} ($\vee$) on the set $\{0,1\}$ for indicators.
Finally, we can assume the existence of a neutral element for every parameter.

The above properties suggest that the structure of a \textit{partially ordered positive monoid} is an appropriate model for a large number of fare-relevant parameters.

\begin{definition}[Partially ordered monoid]
    A \emph{monoid} $(\monoidset,+)$ is a set $H$ together with an associative operation $+$ (called addition) and a neutral element $\neutr\in\monoidset$, i.e., $h+\neutr = h \Forall h\in H$.
        We call $(H,+,\leq)$ a \emph{partially ordered monoid} if $\leq$ is a partial order on $H$ that is translation-invariant with respect to the monoid operation $+$, i.e.,  
        ${h_1 \leq h_2 \Rightarrow h_1+x \leq h_2 +x  \Forall h_1,h_2,x \in H}$.
        If additionally $\neutr\leq h \Forall h \in H$, we call $(H,+,\leq)$ a \emph{partially ordered positive monoid}.
\end{definition}
Note that we can define the cross-product of two partially ordered monoids $(H_1,+_1,\leq_1)$ and $(H_2,+_2\leq_2)$ by $(H_1 \times H_2, +_{12} ,\leq_{1,2})$, 
where $(h_1,h_2) +_{1,2} (i_1,i_2) := (h_1 +_1 i_1, h_2 +_2 i_2)$ and $(h_1,h_2) \leq_{1,2} (i_1,i_2)$ if and only if $h_1\leq i_1$ and $h_2 \leq i_2$ for $h_1,i_1\in H_1$
and $h_2,i_2\in H_2$.
The cross-product of two partially ordered positive monoids is again a partially ordered positive monoid.
This construction allows us to represent all the fare-relevant parameters in Example~\ref{example:mdv_fares} above as a single partially ordered positive monoid $(H,+,\leq)$.

Price-optimal paths can then be found in the following way:
We label each arc $\rarc \in \rarcs$ with a weight in $\monoidset$ representing the relevant parameters on this arc.
The  weight  of a path $p$, denoted by  \wfs$(p)$,  lives in $\monoidset$ as well and can be obtained by summing up the weights of the arcs of $p$.
The ticket $\ticket$ for $p$ and its price can then be derived from $\wfs(p)$ using the rules of the fare structure.
Finding a price-optimal $s,t$-path with $s,t\in V$ can now be achieved by finding the Pareto-set of $s,t$-paths with regard to the partial order of $\pmonoid$.
Here, we can apply a label-setting MOSP algorithm such as, e.g., Martins' algorithm, by using elements of $\monoidset$ as labels and the partial order of $\pmonoid$ to establish dominance between labels
\cite{Parmentier2019a}.
Note, however, that this set will likely still contain many dominated paths with respect to price.
This necessitates the filtering out of superfluous paths in a post-processing step.

\begin{example}[MDV]\label{example:monoid}
For MDV, we can construct a monoid in the following way:
for each city $\city \in \cities$ we define the monoid $(\citymonoid := \{0,1\},\vee,\leq)$ as an indicator whether our path started in $c$ and 
then left the city. Hence, all arcs representing a connection leaving the city carry the weight $1\in \citymonoid$.
Furthermore, we represent the distance traveled by the monoid $(H_{dist} := \N,+,\leq)$, the number of visited stations by $(\stopmonoid:= \N,+,\leq)$,
the set of fare zones by $(\zonemonoid := 2^Z, \cup, \subseteq)$ and finally the transfers by $(\transfermonoid := \{0,1\},\vee,\leq)$.
Price-optimal paths can then be computed by finding the Pareto-set over the monoid
$(\distmonoid\times \stopmonoid \times \transfermonoid \times \zonemonoid \times \prod_{\city\in \cities} \citymonoid, +,\leq)$ and filtering out dominated paths in a post-processing step.
Here, $+$ and $\leq$ are induced from the component monoids.
\end{example}

While the above modeling approach covers a reasonable set of real-world applications, it is not expressive enough to be of much use for complex fare structures.
First, note that it requires an order-preserving relationship between the monoid and the ticket prices.
This assumption does not hold in general: For example, some public transit associations (e.g., in the city of Bremen, Germany, before 2020 \cite{VBN2019}) apply night surcharges on selected lines.
This means that, e.g., in the early morning, it can be beneficial to start a journey later because it becomes cheaper. In particular, it is necessary to capture time in the monoid, but the mapping between time and price does not preserve order.
Second, monoids often cannot express logical pricing conditions: For example, a price depending on the order on which stops are visited cannot be modeled using the monoid-based approach.
Third, shortest path search over the monoid is agnostic to the fare structure and therefore tends to consider unnecessarily large search trees.
In the MDV case, if we already know an $s,t$-path for which, e.g., the ticket $\mdvkl$ is applicable, all partial paths starting in $s$ that require a more expensive ticket can be pruned even though they might not be dominated w.r.t to $\leq$.
To do this, however, we must find a way to quickly obtain the corresponding ticket to labels from $\monoidset$.

Finally, note that the labels of a label-setting MOSP algorithm live in $\monoidset$ and might be quite large.
In most cases, however, it is not necessary to carry the whole label along. 
Consider again the MDV case. If a path has left a city $\city$, all city fares become unavailable and labels for all monoids $\citymonoid, \city \in \cities$, need no longer be considered.
This information, however, remains unavailable to an algorithm using the monoid-based model.

Computational results in Section \ref{sec:results} reveal that purely monoid-based modeling quickly becomes intractable, even when only considering the fare zone monoid $(\zonemonoid,\cup, \subseteq)$.

\subsection{Ticket Graphs and Fare Events}

To overcome the challenges laid out in the previous section, we extend our modeling in two directions:
First, we want to take the path's ticket into account. To do so, we develop a model to represent tickets and their relationships.
Second, we introduce \emph{events} that model logical fare strategies. They also serve to reduce the size of the monoid.
Addressing the first point, we notice that there is a natural progression of the applicable ticket along a path $p$.
If $p$ is short, a short-distance ticket might suffice. When $p$ is extended by adding another stop $v$ at its end, this ticket might no longer be applicable and now, e.g., a zone ticket might apply. When traveling even further, a ticket covering two zones might be needed.
Hence, we can relate tickets to each other via their ability to transition into one another along paths in the routing graph.
We formalize this observation by introducing a \textit{ticket graph}  $\ticketgraph=(\tickets,\tarcs)$
that contains an arc $\tarc=(\ticket_1,\ticket_2)\in \tarcs$ if ticket $\ticket_1$ can transition into ticket $\ticket_2$. 
Transitions depend on the weight $\wfs(p)$ and don't occur at every stop.
Hence, we introduce a ticket transition function that checks $\wfs(p)$ and selects the appropriate ticket from the neighborhood of the ticket of $p$ in $\ticketgraph$.
The ticket graph provides crucial advantages over the purely monoid-based approach:
First, when we want to use price-based target-pruning (cf. Section \ref{sec:tp}), we must compute the current ticket price.
The ticket graph allows doing this by updating tickets along a path, thereby only checking a few transition conditions on the outgoing edges. Without the ticket graph, all  rules of the fare structure would need to be checked whenever a vertex is relaxed.
Second, the ticket graph also carries information about the possible further ticket price development of a partial path: 
This knowledge is useful to design dominance rules.

\begin{example}[Running Example: Ticket Graph for MDV]\label{example:ticketgraph}
\begin{figure}[t]
\centering

\begin{tikzpicture}
	\farenode{z1}{ (4,0) }{$Z1$}{};
    \farenode{z1}{ (4,0) }{$Z1$}{};
	\farenode{z2}{ (4,2) }{$Z2$}{};
	\farenode{z3}{ (4,4) }{$Z3$}{};
	\farenode{z4}{ (2,6) }{$Z4$}{};
	\farenode{z5}{ (4,6) }{$Z5$}{};
	\farenode{z6}{ (6,6) }{$Z6$}{};
	\farenode{m}{ (8,6) }{$M$}{};		
    \farenode{t1}{ (6.75,1.25) }{$C1$}{fill=lightgray};
	\farenode{t2}{ (8, 0 )}{$C2$}{fill=lightgray};
	\farenode{k}{ (8,4) }{$D$}{fill=lightgray};
	\farenode{h}{ (0,2) }{$H$}{};
	\farenode{kh}{ (0,0) }{$D_H$}{fill=lightgray};
	\farenode{l}{ (0,4) }{$L$}{};
	\farenode{kl}{ (0,6 )}{$D_L$}{fill=lightgray};
	
  	\farearc {z1}  {z2} {}{0.2}{0};
  	\farearc {z2}  {z3} {}{0.2}{0};
  	\farearc {z3}  {z4} {}{0.4}{1};
  	\farearc {z4}  {z5} {}{0.2}{0};
  	\farearc {z5}  {z6} {}{0.2}{0};
  	\farearc {z6}  {m}  {}{0.2}{0};	
	\farearc {l}  {z2} {} {0.4} {0};		
	\farearc {h}  {z2} {} {0.4} {0};		
	\farearc {t1} {z1} {} {0.7} {1};			
	\farearc {t2} {z1} {} {0.4} {1};

    \farearc {k}  {z1} {} {0.4} {1};	
    \farearc {k}  {z2} {} {0.5} {1};
    \farearc {k}  {z3} {} {0.25} {1};
    \farearc {k}  {z4} {} {0.25} {1};		
    \farearc {kl} {l} { } {0} {0};	
    \farearc {kh} {h} {} {0.2}{1}; 
    \farearc {kh} {z2} {} {0.3} {0}; 
    \farearc {kl} {z2} {} {0.5} {0};
    \farearc {t2} {k} {} {0.1} {0};
    \farearc {t1} {k} {} {0.1} {1};
\end{tikzpicture}

\caption{Ticket Graph associated with the MDV public transit network. To simplify the presentation, all tickets for city fares are collapsed to $\mdvt1$ and $\mdvt2$ representing the two 
price levels of city fares.
Possible starting tickets are highlighted in light gray.   
\label{fig:mdv_tgraph}}
\end{figure}
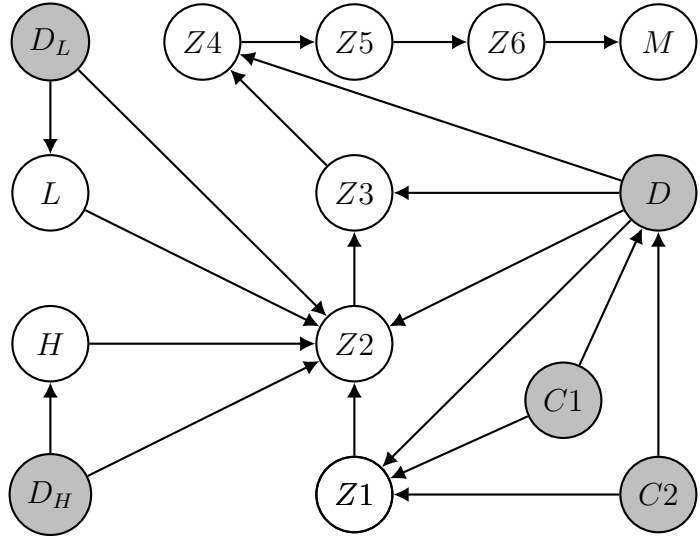
Consider the ticket graph in Figure \ref{fig:mdv_tgraph}. All possible tickets introduced in Example \ref{example:mdv_fares} are represented as nodes.
Whenever, a ticket can transition into another one, we introduce an arc. Note that, e.g., a discounted ticket for Halle $\mdvkh$ can never transition into a Leipzig ticket $\mdvl$.
Now, consider again the path $\station{A}-\station{C}$ in Figure \ref{fig:mdv_plan}. This path requires a ticket covering 
one fare zone. Appending the station $\station{D}$ will require a ticket covering two fare zones.
This is modeled by performing a ticket transition in the ticket graph along the edge $(\mdvz{1},\mdvz{2})$
induced by the ticket transition function for $\mdvz{1}$.
When calculating the ticket price of a path that contains a subpath with ticket $\mdvz{3}$ now only a condition on the number of fare zones needs to be checked to determine whether $\mdvz{3}$ or $\mdvz{4}$ is applicable. Without using the ticket graph, e.g., the applicability of all discounts would need to be recalculated.
Furthermore, when comparing two paths having, say, tickets $\mdvz{2}$ and $\mdvz{3}$, respectively, a dominance check can safely mark the path with ticket $\mdvz{3}$ as dominated even if it is shorter in distance. In the purely monoid-based approach, both paths would be nondominated.
\end{example}

To reduce the dimension of the monoid, note that the parameters determining the ticket of a path can be broadly categorized as \textit{(fare) states}, e.g.,
number of stops, and \textit{(fare) events}, e.g., a transfer or the boarding of a train that requires a surcharge.
Up to now, events were included in the monoid via indicators.
However, this is not strictly necessary: A transfer arc could cause a ticket transition in the ticket graph.
Thereafter, transfers might be ineffectual and hence it is unnecessary to record them in the monoid.
To do so, we need to annotate arcs in the routing graph not only with elements of a  monoid but also with events.
The distinction between \textit{states} and \textit{events} introduces some flexibility to the modeling, as sometimes aspects of a fare structure can be modeled as both. 
However, we naturally aim to keep the monoid as low-dimensional as possible.
Using events, prices may now also depend on logical conditions, e.g., the order of events, which is not possible in the monoid-based approach.

\begin{example}[Running Example: Fare Events for MDV]\label{example:symbols}
The  monoid in Example \ref{example:monoid} was introduced as $(\distmonoid\times \stopmonoid \times \transfermonoid \times \zonemonoid \times \prod_{\city\in \cities} \citymonoid, +,\leq)$.
We can see any transfer as an event $\eventtrans$ occurring on a transfer arc of the routing graph. 
In the same way,  leaving any city $\city \in \cities$ can be seen as an event $\eventcity$ occurring on arcs
crossing the city's borders. Hence, we can shrink the  monoid to $(\distmonoid\times\stopmonoid\times \zonemonoid , +,\leq)$,
reducing the size of a label in a MOSP algorithm by 17 entries for cities and by one entry for transfers.
Finally, we introduce events $\eventhalle$ and $\eventleipzig$ for public transit arcs ending in the special fare zones of Halle and Leipzig, respectively. 
This is purely a design decision, since the information could also be read from the state of $\zonemonoid$.
We obtain a set of events $\fareevents =  \{ \eventcity,\eventtrans,\eventhalle,\eventleipzig, \nullevent \}$ where $\nullevent$ is a dummy event with no effect, see also Example~\ref{example:mdv_cfn}.
\end{example}

\subsection{Conditional Fare Networks}

We now combine the three core ideas of a ticket graph, fare events and modeling with monoids to a formal model of  public transit fare structures.

Again, let $G=(V,A)$ be a routing graph.
Additionally, let $\pmonoid$  be a positive, partially ordered monoid, $\ticketgraph=(\tickets, \tarcs)$ a ticket graph and $\fareevents$ a set of fare events.
We label each arc $\rarc\in\rarcs$ with a weight $\wfa(\rarc)\in \monoidset$ and a fare event $\ef(\rarc)\in\fareevents$.
By collecting these weights and events along a path $p$ in $G$, we build its \textit{fare state} $\state(p)$.
\begin{definition}[Fare State]
A \emph{fare state} $\state\in \tickets\times\monoidset$ is a pair of a ticket $\tf(\state)$ and a weight $\wfs(\state)$.
We write $\statespace := \tickets\times\monoidset$ for the space of all fare states.
\end{definition}
Every vertex $v \in V$ is labeled with an \textit{initial fare state} $\startstate(v) \in \statespace $.
Fare states will serve as path labels for shortest-path algorithms.
They contain all the information necessary to decide domination between paths.
In contrast to common (multi-objective) shortest-path applications, the arc labels $\monoidset\times\fareevents$ live not in the same space as the path labels $\statespace$.

We now want to enable the tracking of fare states along paths in $G$.
To do so, we formalize the notion of the \textit{ticket transition function} on tickets $\ticket \in \tickets$ in the ticket graph $\ticketgraph = (\tickets,\tarcs)$.
A ticket transition function returns the ticket $\ticket_2\in \tickets$  a ticket $\ticket_1 \in \tickets$ transitions into given  an accumulated weight
$\monoidel \in \monoidset$ and a fare event $\fareevent \in \fareevents$. Possible candidates are the neighborhood of $\ticket_1$ in $\ticketgraph$ as well as $\ticket_1$ itself.

\begin{definition}[Ticket Transition Function]
    The ticket transition function $\transfunc: \tickets \times \monoidset \times  \fareevents \rightarrow \tickets$ of $\ticketgraph$ is a
    function that, given a weight $\monoidel\in\monoidset$ and event $\fareevent\in\fareevents$, maps each ticket $\ticket\in\tickets$ into
    its closed out-neighborhood $\delta^+(\ticket) \cup \{ \ticket \}$.
\end{definition}

The definition is intentionally kept as general as possible to capture a large number of possible transition conditions.
We use the notion of ticket transition functions to define the update of a fare state when relaxing an arc of the routing graph.

\begin{definition}[Fare Update Function]\label{def:updatefunction}
	Let $\state \in \statespace$ and $\rarc \in \rarcs$.
	Then,  the \textit{ fare update function}  ${\update: \statespace \times \rarcs \to \statespace}$ is given by 
$\altstate := \update(\state,\rarc)$ with
  \begin{align*}
  \wfs(\altstate)  &:= \wfs(\state) + \wfa(a) \\ 
	\tf(\altstate)  &:= \transfunc(\tf(\state),\wfs(\altstate),\ef(a)).
 \end{align*}    
\end{definition}

The fare state of a path $p=(v_1,\dots,v_n)$ can now be tracked  by letting $ \state_1 := \startstate(v_1)$
and $\state_i := \update(\state_{i-1},(v_{i-1},v_i))  \Forall i = 2 , \dots , n$.
In particular, when calculating $\state_{i}$ we need only consider the out-neighborhood of $\tf(\state_n)$ instead of reevaluating all fare strategies.

Combining all the above definitions, we arrive at the notion of \textit{conditional fare networks} which can precisely describe a fare structure.
\begin{definition}[Conditional Fare Network]
	Let $G=(V,A)$ be a routing graph and let the following be given:
	\begin{enumerate}
	    \item a directed acyclic ticket graph $\ticketgraph = (\tickets,\tarcs)$ with transition function $\transfunc$,
            \item arc weights $\wfa: \rarcs \to \monoidset$  from a partially ordered, positive monoid $\pmonoid$,
            \item arc events $\ef: \rarcs\to\fareevents$,
            \item initial fare states $\startstate: V \rightarrow \statespace$ and 
            \item a price function $\price: \tickets \to \mathbb{Q}_+$ that is monotonously non-decreasing along directed paths in $\tickets$, i.e., if there is 
                    a directed $\ticket_1-\ticket_2$-path in $\ticketgraph$ for $\ticket_1,\ticket_2\in \tickets$, then $\price(\ticket_1) \leq \price(\ticket_2)$. 
                    We write $\price(p)$ instead of $\price( \tf(\state(p)))$ for a path $p\in P$.
	\end{enumerate}
        We call the six-tuple  $\sixtuple$ a \emph{conditional fare network} $\mathcal{N}$ of $G$.
\end{definition}	
Note that cycle-freeness in $\ticketgraph$, the monotonicity condition on $\price$ and the positivity of $H$ ensure that no price-decreasing cycles exist in $G$.
We consider those assumptions natural enough that any reasonable fare structures should satisfy them. 

\begin{example}[Running Example: Conditional Fare Network for MDV]\label{example:mdv_cfn}
We can now give the complete conditional fare network $\farenetwork$ for the fare structure of MDV.
We have already introduced the ticket graph (Example \ref{example:ticketgraph}).
As in Example \ref{example:symbols}, we define the monoid $(\monoidset,+,\leq)$ as $(\distmonoid\times \stopmonoid \times \zonemonoid ,+, \leq)$.
The components  of $(\monoidset,+,\leq)$ are summarized in Table \ref{table:monoid}.
\begin{table}[h]
    \centering
    {
    \begin{tabular}{llllll}
        \toprule
        Name          & Represents    & Ground set   & Operator & Partial Order & Neutral Element \\
        \hline
        $\distmonoid$ & Distance      & $\N$         &  $+$     & $\leq$        & $0$ \\
        $\stopmonoid$ & Stops         & $\N$         &  $+$     & $\leq$        & $0$ \\
        $\zonemonoid$ & Fare Zones    & $2^Z$        &  $\cup$  & $\subseteq$   & $\emptyset$ \\
        \bottomrule
    \end{tabular}
    }
\caption{MDV Fare Monoid}\label{table:monoid}
\end{table}

The set of fare events is $\fareevents = \{\eventcity,\eventtrans,\eventhalle,\eventleipzig,\nullevent\}$. The meaning of these events is summarized in Table \ref{table:events}.
Vertices $v\in V$ in the routing graph can now be annotated with an initial fare state from $\statespace=\tickets\times\monoidset$.
For example, for station $\mathcal{G}$ in Figure \ref{fig:mdv_plan} we have
$\startstate(v) = (\mdvdis,(0,0,\{156\}))$, i.e., we start with the short-distance ticket, and the weight is composed of  {$0$ m} of distance, $0$ visited stops and the fare zone $156$.
Arcs $\rarc \in \rarcs$ are now annotated with weights from $\monoidset$ and events from $\fareevents$.
For example, if $\rarc_1=(\mathcal{F},\mathcal{G})$ had a length of 231 m, it would be annotated 
with $\wfs(\rarc_1) = (231,1,\{233\})$ and $\ef(\rarc_1) = \eventcity$.
The arc $\rarc_2 = (\mathcal{I},\mathcal{H})$ represents a footpath and is annotated with $\wfa(\rarc_2)=(0,0,\emptyset)$
and $\ef(\rarc_2) = \eventtrans$.

\begin{table}[ht]
    \centering
    {
    \begin{tabular}{ll}
        \toprule
        Fare event  & Represents  \\
        \hline
        $\eventcity$ &  Leaving City with City Ticket $\mdvt1$ or $\mdvt2$\\
        $\eventtrans$   &  Transfer \\
        $\eventhalle$   &  Head of Arc is in Halle \\
        $\eventleipzig$   &  Head of Arc is in Leipzig \\
        $\nullevent$ & Nothing \\
        \bottomrule
    \end{tabular}
    }
    \caption{MDV Fare Events}\label{table:events}
\end{table}

Hence, to construct a  conditional fare network for  MDV, we now only need to give the transition functions for $\ticketgraph$.
Let $\fareevent \in \fareevents$ and $\monoidel=(\monoidel_{dist},\monoidel_{stop},\monoidel_{zone})\in \monoidset$. Then, the ticket transition function $\transfunc$ of $\ticketgraph$ is defined by

\footnotesize
 \begin{alignat*}{4}
&\transfunc(\mdvmax, \monoidel, s) &&= \,\,\,\, \mdvmax\\
&\transfunc(Z_i, \monoidel, s) &&= \begin{cases}
                 \mdvz{i+1} & |h_{zone}| = i+1\\
                 \mdvz{i}     & \text{otherwise}  
             \end{cases}&& 
     \transfunc(\mdvdis, \monoidel, s) &&= \begin{cases}
                \mathrlap{\mdvz1}\hphantom{\mdvz{i+1}}  & |\monoidel_{zone}| = 1 \wedge (\monoidel_{dist} > 4 \vee s = \eventtrans )\\
                                            \mdvz{2} & |\monoidel_{zone}| = 2 \wedge (\monoidel_{dist} > 4 \vee s = \eventtrans )\\
                                            \mdvz{3} & |\monoidel_{zone}| = 3 \wedge (\monoidel_{dist} > 4 \vee s = \eventtrans )\\
                 \mdvdis     & \text{otherwise} 
             \end{cases}\\
&    \transfunc(\mdvl,\monoidel,s) &&= \begin{cases}
                \mdvl  & s = \eventleipzig \vee s = \eventtrans\\
                \mathrlap{\mdvz2}\hphantom{\mdvz{i+1}} 
                & \text{otherwise} 
     \end{cases} \quad\quad\quad\quad &&
         \transfunc(\mdvkl,\monoidel,s) &&= \begin{cases}
                \mathrlap{\mdvz2}\hphantom{\mdvz{i+1}} & s \neq \eventleipzig  \wedge  \monoidel_{stop} > 4\\
                \mdvl & s = \eventtrans \vee (s = \eventleipzig \wedge \monoidel_{stop} > 4)  \\
                \mdvkl & \text{otherwise} 
     \end{cases} \\
&     \transfunc(\mdvh,\monoidel,s) &&= \begin{cases}
                H & s = \eventhalle  \vee s = \eventtrans \\
                \mathrlap{\mdvz2}\hphantom{\mdvz{i+1}} & \text{otherwise} 
     \end{cases}&&
             \transfunc(\mdvkh,\monoidel,s) &&= \begin{cases}
                \mathrlap{\mdvz2}\hphantom{\mdvz{i+1}} & s \neq \eventhalle  \wedge \monoidel_{stop} > 4\\
                \mdvh & s = \eventtrans  \vee (s = \eventhalle \wedge \monoidel_{stop} > 4)\\
                \mdvkh & \text{otherwise} 
    \end{cases}\\
&    \transfunc( \mdvt1, \monoidel, s) &&= \begin{cases}
                 \mathrlap{\mdvz1}\hphantom{\mdvz{i+1}} & s = \eventcity  \wedge \monoidel_{dist} > 4\\
                 \mdvdis &  s = \eventcity  \wedge \monoidel_{dist} \leq 4 \\ 
                 \mdvt1     & \text{otherwise} 
             \end{cases}&&
    \transfunc(\mdvt2, \monoidel, s) &&= \begin{cases}
                 \mathrlap{\mdvz1}\hphantom{\mdvz{i+1}} & s = \eventcity \wedge \monoidel_{dist} > 4\\
                 \mdvdis & s = \eventcity  \wedge \monoidel_{dist} \leq 4\\ 
                 \mdvt2 &    \text{otherwise.} 
             \end{cases} \\
\end{alignat*}
\normalsize
For example, the conditions for $\Gamma(D_L, h, s)$ mean that we need to transition to the two zones ticket $\mdvz2$ whenever we leave Leipzig and travel for more than four stops, and that we transition to the standard Leipzig ticket $\mdvl$ whenever we transfer or surpass the four stops limit within Leipzig. In all other cases, we can stick with the discounted Leipzig ticket $\mdvkl$. 
Note that we did not yet cover MDV's overlap areas.
We discuss in Section \ref{par:neutral} why it is best to address these in a preprocessing step.
\end{example}

We can now finally formalize the price-optimal earliest arrival problem.

\begin{definition}[Price-Optimal Earliest Arrival Problem (\POEAP)]\label{def:poeap}
Let a public transportation network be given as a directed graph $G=(V,\rarcs)$
together with a conditional fare network $\sixtuple$ and a time-dependent FIFO travel time function $c(\rarc): I \rightarrow I\ \forall \rarc \in \rarcs$. 
Then, the price-optimal earliest arrival problem (POEAP) asks to find a Pareto-set (w.r.t. price and arrival time) of $s,t$-paths $\pprice \subseteq \stpaths $ in $G$, i.e.,
\begin{align}
\Forall p^* \in \pprice  \Nexists p \in \stpaths :  \price(p) \leq \price(p^*) \land c(p) \leq c(p^*) \land ( \price(p) < \price(p^*) \lor c(p) < c(p^*))  \\
\Forall p\in\stpaths\Exists p^* \in\pprice: \price(p^*)\leq \price(p) \land c(p^*)\leq c(p).
\end{align}

\end{definition}

\subsection{Some Hints on Modeling with CFNs}
In the following, we elaborate on some common features of fare structures and how they can be modeled using conditional fare networks.

\paragraph*{Transfer Penalties, Footpaths, and Surcharges}
When footpaths have no influence on the ticket, they can be modeled as arcs with a weight of $\neutr\in\monoidset$ and an event $\nullevent$ 
that cannot activate a ticket transition.
This way, a footpath does not change the current fare state.
The transition from a footpath to a public transportation vehicle requires some care.
Assume we walk from stop $v_0$ to $v_1$ along arc $a_0=(v_0,v_1)$ to take a vehicle along $a_1=(v_1,v_2)$ to reach $v_2$.
Some fare structures use the number of stops a path touches to calculate prices.
Here, this number would be two.
Counting a stop when relaxing $a_0$ is a mistake if the optimal path would be to continue on foot.
Counting both $v_1$ and $v_2$ when relaxing $a_1$ is also wrong, since this would overcount the number of stops for every journey that reaches $v_1$ via a vehicle.
Hence, the graph model needs to be extended by splitting up stops into vertices for every route and a vertex that is connected to footpaths.
These vertices are then connected via transfer arcs and boarding arcs. 
Placing weights and events different from $\neutr$ and $\nullevent$
on transfer arc allows us to make the applicable ticket dependent on the number of transfers, while events on arcs representing boarding can be used to model surcharges for the boarded route.
For more details on how to build these expanded graphs, we refer to \citet{Disser2008MulticriteriaSP}.

\paragraph*{Overlap areas}\label{par:neutral}
Some fare structures that are based on fare zones contain overlap areas.
Stations in an overlap area can be counted as part of either of its neighboring zones, whichever is cheapest for the costumer. 
This is meant to mitigate sharp price increases for short journeys at fare zone borders.
MDV uses them as well as several other German railway companies (e.g., Verkehrsverbund Bremen/Niedersachsen GmbH \citep{VBB2023Fares}).
	
At a first glance, one might be tempted to represent overlap areas as tickets in the ticket graph. 
A label propagated along a path starting in an overlap area then keeps this ticket until a regular fare zone is picked up along the path and transitions in the zone ticket for this fare zone.
This approach, however, becomes cumbersome when several overlap areas border each other. In this case, a ticket for each combination of overlap areas needs to be introduced.

Alternatively, overlap ares can be incorporated by label duplication:
Assume an overlap area neighbors $n$ fare zones.
We associate each arc $a$ whose $head(a)$ represents a stop in the overlap area with $n$ different weights, one for each fare zone it could possibly be part of.
When settling the vertex in a shortest path search, the current fare state is updated once for each weight, thereby creating $n$ new labels. 
This, however, leads to an increased need for dynamic memory allocation, which should be avoided. 

Hence, we propose simply route duplication as the most convenient model for overlap areas.
Whenever a route of the timetable contains a stop $v$ in an overlap area neighboring $n$ fare zones, we simply introduce $n$ routes each with a single fare zone at $v$.
In the routing graph, this corresponds to introducing parallel arcs with each storing a different fare zone in its weight.
To avoid creating unnecessary duplicates, this is done block-wise, i.e., only for each consecutive sequence of stops along a route that are in the same overlap area.
Hence, overlap areas are taken care of in a preprocessing step and are not represented in the conditional fare network.

\section{CFNs in Routing Algorithms}\label{sec:algo}

Label-setting MOSP algorithms rely on dynamic programming and the subpath optimality condition \citep{Berger_subpath-optimalityof}.
That is, every subpath of an optimal $s,t$-path is in itself an optimal path. 
For \POEAP{}, when 
comparing paths in $G$ naively by the price function $\price$, the subpath optimality condition is usually violated. Consider taking a local detour to  avoid a fare zone: Later on, travelers may be forced to cross the zone due to the infrastructure, turning the locally dominant detour into a suboptimal choice. On the other hand, a locally dominated subpath might still lead to an optimal $s,t$-path. This type of problem persists in CFNs: the transition between tickets depends on the weights and events already collected, but also on the structure of the reachable ticket graph.
Example \ref{example:dom} highlights that problems can already arise even in simple cases.
	
\begin{example}[Label Dominance in Figure \ref{ref:comp}]\label{example:dom}
		
\begin{figure}[ht]
	\centering
 \begin{minipage}{\linewidth}
 \begin{minipage}{\linewidth}
	\centering
\subfloat[Routing Graph]{
	\begin{tikzpicture}[scale=0.7]
	\tikzset{nodesti/.style={draw,circle,minimum size=0.8cm}}
	\node[nodesti] (v1) at (-2,0){\small$v_1$};	
	\node[nodesti] (v2) at (2,2){\small$v_2$};	
	\node[nodesti] (v3) at (2,-2){\small$v_3$};	
	\node[nodesti] (v4) at (6,0){\small$v_4$};	
	\node[nodesti] (v5) at (12,0){\small$v_5$};
        \draw[->] (v1) -- (v2) node[midway,align=left,above,xshift=-1cm,yshift=0.1cm] {\small$\wfa(v_1,v_2)  = 0$\\$\ef(v_1,v_2) = s_0$};
        \draw[->] (v1) -- (v3) node[midway,align=left,below,xshift=-1cm,yshift=-0.1cm] {\small$\wfa(v_1,v_3) = 0$\\$\ef(v_1,v_3) = s_0$};
	\draw[->] (v2) -- (v4) node[midway,align=left,above,xshift=1cm,yshift=0.1cm] {\small$\wfa(v_2,v_4)   = 1$\\$\ef(v_2,v_4) = s_1$};
	\draw[->] (v3) -- (v4) node[midway,align=left,below,xshift=1cm,yshift=-0.1cm] {\small$\wfa(v_3,v_4)  = 2$\\$\ef(v_3,v_4) = s_2$};	
	\draw[->] (v4) -- (v5) node[midway,align=left,below,yshift=-0.3cm] {\small $\wfa(v_4,v_5)            = 2$\\$\ef(v_4,v_5) = s_3$};
	\end{tikzpicture}
	}
	\end{minipage} 
 \begin{minipage}{\linewidth}	
	\centering	
 \begin{minipage}{.45\linewidth}
	\centering
\subfloat[Ticket Graph]{
	\begin{tikzpicture}[scale=0.7]
	\tikzset{nodesti/.style={draw,circle,minimum size=0.8cm}}
	\node[nodesti] (A) at (0,-0.5){\small$A$};	
	\node[nodesti] (B) at (3,1){\small$B$};
	\node[nodesti] (C) at (6,1){\small$C$};
	\node[nodesti] (D) at (3,-2){\small$D$};
	\node[nodesti] (E) at (6,-2){\small$E$};
	\draw[->] (A) -- (B) node[midway,above,xshift=-0.3cm] {\small$ \mathbbm{1}_{ \{ s =s_1 \} }  $};
	\draw[->] (B) -- (C) node[midway,above] {\small$  \mathbbm{1}_{ \{ s =s_3 \} }  $};
	\draw[->] (A) -- (D) node[midway,below,xshift=-0.3cm] {\small$\mathbbm{1}_{ \{ s =s_2 \} }  $};							
	\draw[->] (D) -- (E) node[midway,below] {\small$\mathbbm{1}_{ \{ s =s_3 \} }  $};							
	\end{tikzpicture}	
}
\end{minipage}
\begin{minipage}{.54\linewidth}
\centering
\subfloat[Ticket Graph]{
\makebox[.51\linewidth]{ 
	\begin{tikzpicture}[scale=0.7]
	\tikzset{nodesti/.style={draw,circle,minimum size=0.8cm}}
	\node[nodesti] (A) at (0,-0.5){\small$A$};	
	\node[]  at (3,1){};
	\node[nodesti] (B) at (4,1){\small$B$};
	\node[] at (3,-2){};
	\node[nodesti] (C) at (4,-2){\small$C$};
	\draw[->] (A) -- (B) node[midway,above,xshift=-0.5cm] {\small $\mathbbm{1}_{ \{ s = s_3 \wedge  h \leq 3 \}} $};
	\draw[->] (A) -- (C) node[midway,below,xshift=-0.5cm] {\small $ \mathbbm{1}_{ \{ s = s_3 \wedge  h > 3 \}}$};						
	\end{tikzpicture}}}
\end{minipage}
\end{minipage}\end{minipage}
\medskip 
\caption{Example of a routing graph \textbf{(a)} with two possible conditional fare networks \textbf{(b)} and \textbf{(c)}.
For both networks, the underlying partially ordered monoid is $(\mathbb{R},+,\leq)$, 
the fare events are  $\fareevents = \{s_0,s_1,s_2,s_3\}$ and the initial fare state for all vertices $v_i$ with $i= 1,\dots,5$ is $\startstate(v_i) = (A,0)$.  
We set prices for the tickets as $\price(A) = 0$, $\price(B) = 2$, $\price(C) = 3$, $\price(D) = 1$ and $\price(E) =5$.
The value of the transition function $\transfunc$ for a given weight $\monoidel$ and event $\fareevent$ is given via indicator functions on the fare arcs.
Using the ticket graph \textbf{(b)}, the upper $v_1,v_5$-path yields ticket $C$, while the lower path yields ticket $E$.
Using ticket graph \textbf{(c)},
the upper path yields ticket $B$, the lower path yields ticket $C$.}\label{ref:comp}
\end{figure}
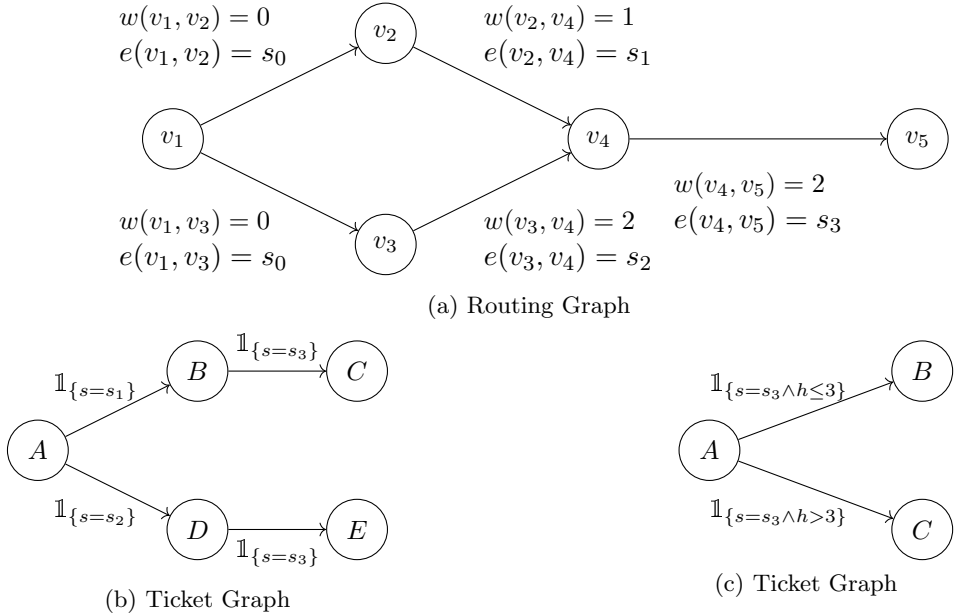 

Consider the routing graph \textbf{(a)} together with the  conditional fare network \textbf{(b)}.
Examining the paths $p_1 = (v_1,v_2,v_4)$ and $p_2=(v_1,v_3,v_4)$, we find their respective  fare states are $f(p_1) = (B,1)$ and $f(p_2) = (D,2)$.
Extending them by $v_5$ to $p_{1}'$ and $p_{2}'$ yields $f(p'_1) = (C,3)$ and $f(p'_2) = (E,4)$.
Comparing fare states by price would indicate that $p_1$ could be pruned at $v_4$ since $\price(B) > \price(D)$.
This is a suboptimal choice as $p'_1$ dominates $p'_2$ since $\price(C) < \price(E)$.
Hence, price cannot be used as dominance criterion for fare states.
A natural alternative would be to use the partial order defined by paths in the ticket graph, instead.
A ticket $\ticket_1$ then dominates a ticket $\ticket_2$ if there is a $\ticket_1,\ticket_2$-path.
This would render the tickets $B$ and $D$ and the tickets $C$ and $E$ mutually incomparable.
The idea, however, comes with problems of its own.
To see this, consider the conditional fare network \textbf{(c)}.
At $v_4$, we have $\state(p_1) = (A,1)$ and $\state(p_2) = (A,2)$ and hence both paths are equivalent and 
it would be sensible to keep only one of them based on the relation between $\wfs(\state(p_1))$ and $\wfs(\state(p_2))$.
By relaxing $(v_4,v_5)$,  we obtain $\state(p'_1) = (B,3)$ and $\state(p'_2) = (C,4)$, which are incomparable, i.e., the fare states of $p'_1$ and $p'_2$ diverged from comparable to incomparable.
Consequently, any dominance rule pruning either $p_1$ or $p_2$ would be defective.
\end{example}

To mitigate these and similar problems, we might assume a general incomparability of fare states. This comes down to enumerating all $s,t$-paths and simply sorting them by price.
However, in a sensibly designed fare structure, it is usually clear which ticket is better, and taking a cheaper subpath should usually not turn out more expensive overall.
In the remainder of this section, we propose a more tailored approach.
It bases domination rules on path relationships but adds exceptions to cover cases in which it is not safe to do so.
    
\subsection{Dominance for Fare States}

We want to define a partial order for fare states that restores subpath optimality while not relaxing dominance too generously. 

To do so, we partition the ticket set $\tickets$  into three disjoint \textit{comparability groups}: $C_F$ (full comparability), $C_P$ (partial comparability), $C_N$ (no comparability).
Based on the partition $C =(C_F,C_P,C_N)$, we define the partial order.
\begin{definition}[Comparability of Fare States]\label{def:compstates}
    Let $\state_1 = (\ticket_1,\monoidel_1)$, $f_2=(\ticket_2,\monoidel_2)$ be fare states. We say $\state_1 \leqc \state_2$ if and only if
    $\ticket_1 \notin C_N$, $ \monoidel_1 \leq \monoidel_2 $ and 
    \begin{alignat}{5}
    &\ticket_1 = \ticket_2 &&\text{if } \ticket_1\in C_P\\
    & \Exists \ticket_1,\ticket_2\text{-path in $\ticketgraph$}\quad &&  \text{if } \ticket_1 \in C_F.
    \end{alignat}
    If $\state_1\leqc \state_2$ and either $\monoidel_1 < \monoidel_2$ or $\ticket_1 \neq \ticket_2$, 
    we say that $\state_1$ is \emph{strictly less} than $\state_2$, i.e., $\state_1 \lc \state_2$.
\end{definition}

We denote by $\pstate$ the set of all Pareto-optimal paths with respect to $\leqc$, i.e.,
\begin{equation}\label{eq:mono}
p^*  \in \pstate \Rightarrow \Nexists s,t\text{-path } p : \state(p)\lc \state(p^*).
\end{equation}

We call paths in  $\pstate$ \textit{state-optimal}. They form a superset of the set of price-optimal paths.
MOSP algorithms on graphs with weights from partially ordered monoids rely on the monoid operation being translation-invariant with respect to the partial order.
Similarly, for CFN's, we need the update function to be monotone along all arcs $\rarc \in \rarcs$, i.e.,
\begin{equation}\label{eq:updatemono}
\Forall \state_1,\state_2\in \statespace: \state_1 \leqc \state_2 \Longrightarrow  \Forall \rarc \in \rarcs : \update(\state_1,\rarc) \leqc \update(\state_2,\rarc). 
\end{equation} 
This condition is enough to ensure that a weaker form of subpath optimality holds.
	
\begin{proposition}[Weak Subpath Optimality]\label{proposition:subpathpotimality}
    Let $G=(V,A)$ be a routing network and \mbox{$\farenetwork = \sixtuple$} be its conditional fare network. 
    Let $p^* \in \pstate $ be a state-optimal \mbox{$s$,$t$-path}  in $G$ for some $s,t\in V$.
    Then, there is a path $p'= (s=v_0,v_1, \dots,v_{n-1},v_n = t) \in \pstate$ with $\state(p^*)=f(p')$, such that every subpath $p''=(v_0,\dots,v_l)$, $l< n$, of $p'$ is a state-optimal $v_0,v_l$-path.
\end{proposition}
\begin{proof}
    Let $p^*=(s=v_0,v_1,\dots,v_{n-1},v_n=t)\in \pstate$ be a state-optimal $s,t$-path with fare states $(\state_0,\dots,\state_n)$.
    Assume there is a $s,t$-path \mbox{$\tilde{p} = (s = u_0,u_1,\dots,u_{l-1}=v_{n-1},u_l= t)$} with fare states $(\altstate_0,\altstate_1,\dots, \altstate_l)$ and $\altstate_0=\state_0$. By \eqref{eq:updatemono}, we can choose $\tilde{p}$ to be a simple path.
    Let $k$ be the largest integer such that $v_{n-k} = u_{l-k}$, i.e., the paths $(v_{n-k},\dots,v_{n})$ and $(u_{l-k}, \dots u_{l})$ are equal.
    Now assume $\tilde{\state}_{l-k-1} \lc \state_{n-k-1}$.
    By definition, $\state_{n-k} = \update(\state_{n-k-1},(v_{n-k-1},v_{n-k}))$ and 
    $\altstate_{l-k} = \update(\altstate_{l-k-1},(u_{l-k-1},u_{l-k}))$.
    We apply \eqref{eq:updatemono} to obtain $\tilde{f}_{l-k} \leqc f_{n-k}$.
    By repeating the process for $i\in\{k-1,\dots,0\}$, we find 
    $\altstate_{l} \leqc \state_{n}$. Since $p^*$ was state-optimal, it follows that $\altstate_{l} = \state_{n}$,
    and consequently, $p$ is also state-optimal.
    Since the number of paths in $G$ is finite, we can repeat this procedure to find the path $p'$.
\end{proof}
Proposition \ref{proposition:subpathpotimality} does not imply that every subpath of a state-optimal path is state-optimal.
We can, however, discard all state-optimal paths without this property since a path with an equal fare state still remains in $\pstate$.
Hence, label-setting MOSP algorithms  can still be applied.
	
\subsection{The Comparability Partition}
    
In choosing $\fullcomp$, $\partcomp$ and $\nocomp$, there is some degree of freedom.
We want $\fullcomp$ to be as big and $\nocomp$ as small as possible while still fulfilling \eqref{eq:updatemono}.
It is clear that the best choice does not only depend
on the ticket graph $\ticketgraph$ and the transition function $\transfunc$, but also on 
$G$ and the arc weights and events.
Such an approach, however, mostly likely requires extensive computations on $G$.
We propose a solution that depends only on $\ticketgraph$ and $\transfunc$ and needs no recomputation when changes in the  routing network occur.

First, we introduce some notation.
If there is a directed path in $\ticketgraph$ between $\ticket_1,\ticket_2\in \tickets$, we write $\ticket_1 \patheq \ticket_2$. This includes the case  $\ticket_1=\ticket_2$.
The \emph{reach} $\reach(\ticket)$
of a vertex $\ticket \in \tickets$ is the subgraph induced by all vertices reachable from $\ticket$, i.e.,
$
\reach(\ticket) := \ticketgraph[\{ k \in \tickets:  \ticket \patheq k\}].
$
\begin{definition}[No-overtaking Property]
    Let $\ticket \in \tickets$ be a ticket. 
    We say its reach $\reach(\ticket)$ has the \emph{no-overtaking property} if for all tickets $k,l \in \reach(\ticket)$ with $k\patheq l$ and $(\monoidel,\fareevent) \in \attributespace$ it
    holds that
    \begin{equation}
    \Forall \bar{\monoidel} \in \monoidset:  \monoidel \leq \bar{\monoidel} \Longrightarrow \transfunc(k, \monoidel,\fareevent) \patheq \transfunc(l,\bar{\monoidel},\fareevent). 
    \label{eq:noover}
    \end{equation}
\end{definition}

The no-overtaking property bears some resemblance to the FIFO (first-in, first-out) property:
A worse fare state, i.e., a worse weight or ticket, cannot give rise to a better fare state when relaxing the same arc in the routing graph.
Note that the no-overtaking property has to be fulfilled not only for the neighborhood of a ticket $\ticket$ but for the reach $\reach(\ticket)$.
Subgraphs with the no-overtaking property allow for the strictest domination rules. We use them as comparability group $\fullcomp$.
	
\begin{definition}[Comparability Partition]\label{def:compPar}
    Let $G=(V,\rarcs)$ be a routing network with a CFN ${\farenetwork = \sixtuple}$. 
    We define
    \begin{align}
    \fullcomp &:=  \{ \ticket \in \tickets  : \reach(\ticket) \text{ traceable and has the no-overtaking property}  \}\\
    \partcomp &:=  \{ \ticket \in \tickets \backslash \fullcomp: \Forall k \in \reach(\ticket)
    \Forall  \fareevent \in \fareevents \Forall \monoidel_1,\monoidel_2 \in \monoidset  :\transfunc(k,\monoidel_1,\fareevent) = \transfunc(k,\monoidel_2,\fareevent)  \}\\
    \nocomp &:=    \{ \ticket \in \tickets \backslash(\fullcomp\cup \partcomp) \}.
    \end{align}
\end{definition}

It is not enough to fulfill \eqref{eq:noover} for $\ticket\in\tickets$ to be in the set $\fullcomp$.
Its reach $\reach(\ticket)$ has also to be \emph{traceable}, i.e., contain a Hamiltonian path.
This condition is needed to avoid the divergence seen in Example \ref{example:dom}. 
If a ticket has non-traceable reach or does not have the no-overtaking property, it is placed in $\partcomp$. 
For tickets $\ticket \in \partcomp$, the transition functions of tickets $k\in\reach(\ticket)$ must be independent of $\pmonoid$.
This, again, is necessary to ensure that comparable fare states do not diverge in an incomparable state after an update, i.e., all tickets that can be reached from a ticket in $\fullcomp$ themselves need to be in $\fullcomp$.
All remaining tickets are added to $\nocomp$.
Fare states containing tickets from $\nocomp$ can never be dominated. 

\begin{example}[Dominance for MDV Fares]
    In the graph in Figure \ref{fig:mdv_tgraph}, all nodes have traceable reach, and it is easy to verify that the no-overtaking property does indeed hold for all tickets.
    Hence, we can set the comparability partition to $\fullcomp = \tickets$, $\partcomp = \nocomp = \emptyset$.
\end{example}

\begin{example}[Dominance for Example \ref{example:dom}]
    For Ticket Graph \textbf{b)}, we have 
    \mbox{$\fullcomp = \{B,C,D,E\}$}, $\partcomp=\{A\}$ and $\nocomp=\emptyset$.
    For Ticket Graph \textbf{c)}, we have that
    $\fullcomp=\{B,C\}$, $\partcomp = \emptyset$ and $\nocomp=\{A\}$.
\end{example}

\begin{proposition}[Monotonicity  of the Comparability Partition] \label{proposition:mono}
    The partial order $\leqc$ defined by  Definitions \ref{def:compstates} and  \ref{def:compPar} fulfills the monotonicity condition \eqref{eq:updatemono}.
\end{proposition}

\begin{proof} 
Let $\rarc \in \rarcs$ and $\state_1,\state_2\in \statespace$ such that $\state_1 \leqc \state_2$. 
    For $i\in\{1,2\}$, we write $\altstate_i := \update(\state_i,\rarc)$ , i.e,
    $\wfs(\altstate_i) = \wfs(\state_i) + \wfa(\rarc)$ and 
    $\tf(\altstate_i) = \transfunc(\tf(\state_i),\wfs(\altstate_i), \ef(\rarc))$.
    By positivity of the monoid $\pmonoid$, $\wfs(\state_1) \leq \wfs(\state_2)$ directly implies $\wfs(\altstate_1) \leq  \wfs(\altstate_2)$.
    It remains to show that $\tf(\altstate_1) \patheq \tf(\altstate_2)$.
    To do so, we need to distinguish the cases  
    $\tf(\state_1)\in \partcomp$ and $\tf(\state_1)\in \fullcomp$.

    First, assume that $\tf(\state_1)\in \partcomp$ and hence
    $\tf(\state_1) = \tf(\state_2)$.
    By the definition of $\partcomp$, we obtain 
    \[\tf(\altstate_1) = \transfunc(\tf(\state_1),\wfs(\altstate_1), \ef(\rarc)) =  \transfunc(\tf(\state_2),\wfs(\altstate_2), \ef(\rarc))= \tf(\altstate_2).\]
    Thus, $\tf(\altstate_1) = \tf(\altstate_2)$.
    Note that the definitions of $\partcomp$ and $\fullcomp$ imply that 
    $\tf(\altstate_1) \in \partcomp \cup \fullcomp $ 
    since 
    $\tf(\altstate_1) \in \reach(\tf(\state_1))$ and hence $\altstate_1 \leqc \altstate_2$.
    Now, assume $\tf(\state_1) \in \fullcomp$. Note that $\reach(\tf(\state_1)) \subset \fullcomp$. 
    This allows us to apply \eqref{eq:noover} to obtain 
    \[ 
         \tf(\altstate_1) = \transfunc(\tf(f_1),\wfs(\altstate_1), \ef(\rarc)) \patheq  \transfunc(\tf(\state_2),\wfs(\altstate_2), \ef(\rarc)) =  \tf(\altstate_2),
    \]
    which concludes the proof.
\end{proof}
	
Propositions  \ref{proposition:subpathpotimality} and \ref{proposition:mono} allow us to apply label-setting MOSP algorithms to POEAP using the comparability partition from Definition \ref{def:compPar}. However, we obtain only the set of state-optimal paths.
It remains to show that this set contains the cheapest path. 
	
\begin{proposition}[Correctness]\label{prop:correctness}
    Let $\price^* := \min_{\stpaths}\price(p)$. Then, there is at least one $s,t$-path $p^*$ with  $\price^* = \price(p^*)$ and
    $p^* \in \pstate$. \end{proposition}

\begin{proof}
Consider a path $p\in\stpaths$ with $\price(p) = \price^*$.
If there is $p' \in \pstate$ with 
\mbox{$\tf(\state(p')) = \tf(\state(p))$}, we are done. If not, all such paths must be dominated w.r.t. to $\lc$ and hence there is a path $p' \in \pstate$ with
$\tf(\state(p')) \rightarrow \tf(\state(p))$.
This implies $\price(\tf(\state(p'))) \leq \price(\tf(\state(p)))$ and hence a path of the same price as $p$ is present in $\pstate$.
\end{proof}

\section{Price-Optimal RAPTOR}\label{sec:psrap}

In this section, we will discuss how to use conditional fare networks to implement a price-optimal version of the multi-criteria RAPTOR  algorithm (abbr. \MCRAP)  \citep{DBLP:journals/transci/DellingPW15}. 
Since \RAP{} implicitly optimizes the number of trips, we obtain an multi-criteria algorithm that optimizes for travel time, number of trips and price.
So far, we have presented our framework in a graph-based context. \RAP, however, does not use a graph model but works directly on the timetable. The adaption for \RAP{} is straightforward.

This section is structured as follows.
A review of the \RAP{} algorithm is provided in Section \ref{sec:rapintro}.  
Then, in Section \ref{sec:fareinrap}, we lay out how \MCRAP{} can be modified to use conditional fare networks for price optimization. 
In Section \ref{sec:restricted}, we briefly recap how improvements in run times can be achieved by calculating a restricted Pareto-set using the recently introduced Bounded-RAPTOR-algorithm 
(\BMRAP) \citep{doi:10.1137/1.9781611975499.5}.
The algorithm excludes all journeys that need significantly more transfers or take significantly more time than the journeys found with an (earliest arrival) \RAP\ query. 
Finally in Section \ref{sec:speedups}, we introduce two speed-up techniques that are tailored to our application. 

Sections \ref{sec:rapintro} and \ref{sec:restricted} give succinct summaries of the  \RAP{} and \BMRAP{} algorithms. 
For a thorough presentation, see the original research in \citet{DBLP:journals/transci/DellingPW15} and \citet{doi:10.1137/1.9781611975499.5}, respectively.
We provide the pseudo-code of \BMRAP{} in \ref{app:mcrap}.

\subsection{Multi-Criteria Search with \MCRAP}\label{sec:rapintro} 
We largely adhere to  the notation  of \citet{DBLP:journals/transci/DellingPW15} albeit with minor modifications to avoid the reuse of variables.

The \RAP{} algorithm does not use a graph model, but works directly on the timetable. 
A timetable is a tuple $\TimeTable$ consisting of a period of operation $\TTPI$, a set of stops $\TTP$, a set of routes $\TTR$, a set of trips $\TTT$ and a set of footpaths $\TTF$. 
A \textit{stop} $p\in \TTP$ is a location where a vehicle can be boarded or exited. 
Each stop $p$  has  a (possibly zero) transfer time $\transtime(p) \in \N$ that is applied whenever a vehicle is boarded at $p$. 
A trip $\raptrip\in\TTT$ is a sequence of stops together with arrival and departure times $\arr(\raptrip,p)$ and $\dep(\raptrip,p)\in\TTPI$.
A \textit{route} is a set of trips, where all trips  have the same sequence of stops and no trip overtakes another one.
We denote the sets of stops and trips of a route $r$ by $\TTP(r)$  and $\TTT(r)$, respectively.
Finally, a footpath $(p_1,p_2,l)\in \TTF$ is a pair of stops $(p_1,p_2)$ combined with a walking time $l$. 

\RAP{} operates in rounds $k = 1,\dots,K$ on $\TT$ with $K\in \N\cup\{\infty\}$ the maximum number of transfers to be considered.
It maintains arrival time labels $\arr(k,p)$ for every $p\in\TTP$ and every round $k\in 1,\dots,K$. Every entry of $\arr$ is initialized to $\infty$. 
Each round begins with a set of marked stops. All routes touching these stops are collected and then processed in an arbitrary order. 
When processing a route $r\in \TTR$, \RAP{} begins with the first marked stop and from there iterates through the stops in order of travel. For each stop $p\in \TTP(r)$, \RAP{} finds the earliest trip $\raptrip$ that can be taken at $p$ after $\arr(k-1,p) +\transtime(p)$.
In the same sweep, the arrival times of $\raptrip$ are used to update $\arr(k,p)$. All stops whose arrival times improved over $\arr(k-1,p)$ are marked.

In a second step, a footpath search is performed. All footpaths $(p_1,p_2,l)$ starting at a marked stop are processed in an arbitrary order, and labels in $\arr(\cdot,p_2)$ are updated accordingly. 
Again, each stop $p$ with an improved arrival time $\arr(k,p)$ is marked for the next round.
Note that this requires the footpath set to be transitively closed. 
The algorithm terminates after $K$ rounds or when no more stops can be marked.

\RAP{} can be modified slightly to allow for multi-criteria search.
Instead of only the arrival times $\arr(k,p)$, \MCRAP{} now maintains a \textit{label bag} $B_k(p)$ for every stop $p$ and round $k$. Each label $L\in B_k(p)$ contains an entry for every optimization criterion.
When processing a route $r\in \TTR$ at a starting stop $p_s$, a route bag $B_r$ is created, and all labels from $B(k-1,p_s)$ are updated with $\transtime(p)$ and copied into $B_r$. 
Each label in $B_r$ is associated with a trip $\raptrip \in \TTT(r)$. 
At each stop $p\in \TTP(r)$ after $p_s$, \MCRAP\ updates all labels in $B_r$, merges $B_r$ into $B_k(p)$ and finally merges all labels from $B(k-1,p)$ into $B_r$ and assigns a trip to them. 
In each step, dominated labels are removed.  

\subsection{Using Conditional Fare Networks in \MCRAP\ } \label{sec:fareinrap}
Adapting \MCRAP\ to incorporate fares is now  fairly straightforward but requires several modifications to $\TT$.
For every trip $\raptrip\in \TTT$, we  additionally store two pairs of weights and events for each stop $p$.
The first pair $\tripweight(\raptrip,p)$, $\tripevent(\raptrip,p)$ is considered when reaching the stop $p$ while iterating along $\raptrip$.
The second pair $\transweight(\raptrip,p)$, $\transevent(\raptrip,p)$ is picked up when boarding $\raptrip$ at $p$.
This distinction is necessary since there is no direct equivalent in $\TT$ to the transfer arcs used in the graph-based setting to model, e.g., surcharges.      
Furthermore, the definition of a route needs a slight adjustment: A route $r\in \TTR$ is a set of trips where all trips have the same sequence of stops 
\textit{and the same weights and events}, i.e., for $i=1,2$, $\wfa_i(\raptrip_1,p) = \wfa_i(\raptrip_2,p)$ and $\ef_i(\raptrip_1,p) = \ef_i(\raptrip_2,p)$ for all $\raptrip_1,\raptrip_2\in \TTT(r)$ and $p \in \TTP(r)$, and no trip overtakes another one.

Now, a label $L= (\raptime,\state)$  in a label bag $B_k(p)$ consists of an arrival time $\raptime$ and a fare state $\state$. Labels in route bags additionally hold the current trip $\raptrip$.
When updating a label $L = (\raptime,\state,\raptrip)$  from route bag $B_r$ at a stop $p$, the arrival time $\raptime$ is updated to $\arr(\raptrip,p)$ and the fare state is updated 
with $\transweight(\raptrip,p)$ and $\transevent(\raptrip,p)$ as in Definition \ref{def:updatefunction}.
Update steps that are associated with transfers are performed whenever labels are merged into $B_r$. 
Here, $\transweight(\raptrip,p)$ and $\transweight(\raptrip,p)$ are used for updating.
Dominance of labels is checked according to the theory developed in Section \ref{sec:algo} while also taking arrival times into account. 
Since walking is usually free of charge, fare states do not need to be updated in the footpath stage. Hence, footpaths are also not enriched with fare information.

Using \MCRAP, we obtain the Pareto-set $\jstate$, optimizing for arrival time, number of trips and fare state.
The smaller set $\jprice \subseteq \jstate$, optimizing for price instead of fare state, can be calculated in a post-processing step.

\subsection{Restricted Pareto-Sets}\label{sec:restricted}
By design of fare structures, the cheapest path is often among the fastest, as detours are penalized by increases in price, arrival time and transfers.
At other times, a negligible reduction in price might be achievable at the expense of a significant increase in travel time. Such journeys are unlikely to be chosen by a traveler.
Hence, it appears beneficial for a price-optimal search to prune all labels that are worse by a certain margin (regarding both arrival time and transfers) than the results of a normal \RAP\ query.
This is achieved by the following pruning schemes,  first introduced  by \citet{doi:10.1137/1.9781611975499.5} for general multi-objective search with \RAP{}.

Let  $\janch$ be the Pareto-set of all \textit{anchor journeys} found with a \RAP\ query, i.e., optimizing only for arrival time and number of transfers. 
We denote the arrival time of a journey $J$ by $\arr(J)$ and its number of trips by $\tr(J)$.
We aim to calculate a restricted Pareto-set $\jr$ with  $\janch  \subseteq \jr \subseteq  \jstate$ of journeys that do not have  a significantly higher arrival time or number of trips than some journey from $\janch$.
Let $\asig \in \R^+$ and $\tsig\in \N^+$ be the maximal acceptable slacks for arrival time and number of transfers, respectively.
Then, we define
\[
    \jr  := \{ J \in \jstate \,  |\, \exists  \sjanch \in \janch \text{ such that }  \arr(J) \leq \arr(\sjanch) + \asig \text{ and } \tr(J) \leq \tr(\sjanch) + \tsig  \}.
\]

To obtain a two-stage pruning scheme, we can first run a normal \RAP\ query. 
The labels obtained in this first stage can then be used to prune the multi-criteria search. 
Let $\raptime_k$ be the optimal arrival time at the target stop $p_t$ in round $k$ of the first stage (computed with \RAP). During round $k$ of \MCRAP, we prune every label that has an
arrival time $\raptime$ with $\raptime >  \raptime_k +\asig$.
This pruning scheme is called \TargetBMRAP. Note that \TargetBMRAP\ works on the assumption that $\tsig = \infty$. 
Also, the bound is not tight for all stops other than $p_t$.

The set $\jr$ can be computed  with the more involved \TightBMRAP{}.  Here, three rounds are performed.
As for \TargetBMRAP{}, we first perform a normal \RAP\ search, obtaining the Pareto-set $\janch$.
Then, multiple reverse \RAP\ queries are performed to build bounds at all stops $p\in \TTP$.
The third stage is the actual \MCRAP\ round using the previously computed bounds for pruning.

In the following, we describe the second and third stages in more detail.
Let $m := K + \tsig$, where $K$ is the maximum number of trips in any journey $ \sjanch \in \janch$.
A backward \RAP\ search with starting time $\arr(\sjanch) + \asig$ and $n_{\sjanch} : = \tr(\sjanch) + \tsig$ rounds is performed for every journey $\sjanch \in \janch$.
The backward search works on departure times instead of arrival times and transfer times are not applied when boarding a vehicle but instead when disembarking.

Each backward search computes latest departure times $\backdep(\sjanch,k,p)$ such that $p_t$ can still be reached earlier than $\arr(\sjanch) + \asig$ while using at most $k$ more trips onward from $p$.
It is possible that $\backdep(\sjanch,k,p)$ remains at its initialization value of $-\infty$.
Using $\backdep(\sjanch,n_{\sjanch}-k,p)$ for pruning labels in round $k$ of a forward \MCRAP\ search computes 
\[\{ J \in \jstate \,  |\, \arr(J) \leq \arr(\sjanch) + \asig \text{ and } \tr(J) \leq \tr(\sjanch) + \tsig  \}.\]

Carefully overlapping the labels $\backdep(\sjanch,k,p)$ results in a set of labels $\backdep(k,p)$ with $k= 1, \dots,m$, i.e.,
\[ 
    \backdep(k,p) := \max_{
    \substack{\sjanch \in \janch: \\ k \geq m - n_{\sjanch}}}
    \{\backdep(\sjanch,k - m + n_{\sjanch},p)  \}.
\]
\begin{figure}[t]
    \centering
    \begin{tikzpicture}
        	\begin{axis}[
		xlabel=Num. of Trips,
		ylabel=Arrival Time,
                xmin=0,
                xmax=6.5,
                ymin=0,
                ymax=65,
                xtick={1,2,3,4,5,6},
                ytick={10,20,30,40,50,60},
                grid=major,
                ]
        \path[fill=gray!17] (4,10) rectangle ++(1,30);
        \path[fill=gray!17] (1,30) rectangle ++(1,30);
        
        \addplot [only marks,mark=*, mark size=5pt] coordinates { (1,30)
                                                    (4,10) };
        \addplot [only marks,mark=square*, fill, mark size=4pt] coordinates { (3,30)
                                                     };
        \addplot [only marks,mark=triangle*, fill, mark size=5pt] coordinates { (4,35)
                                                     };
        \addplot[color=black, ultra thick, dotted] coordinates {
                    (5,0)
                    (5,10)
                    (5,40)
                    (2,40)
                    (2,60)
                    (0,60)
                    };
        \addplot[color=black, ultra thick, dashed] coordinates {
                    (1,70)
                    (1,30)
                    (4,30)
                    (4,10)
                    (7,10)
                    };
                \end{axis}
    \end{tikzpicture}
    \caption{Illustration of the  solution space of a \TightBMRAP{} search with $\asig=30 \textrm{ min}$ and $\tsig = 1$.
        The circle marks represent the anchor journeys from $\janch$. 
        \TightBMRAP{} prunes all journeys to the right of the dotted line spanned by those journeys.
        The area to the left of the dashed line contains no Pareto-optimal journeys.
        The journeys from $\jstate$ that fall into the area enclosed by the dashed and dotted lines form $\jr$.
        The light gray area forms $\jdelling$.
        Note that the journey marked by the square mark is in $\jr$ but not in $\jdelling$.
        The journey represented by the triangle mark is in $\jdelling$ even though it uses more trips and has a later arrival time.
        }\label{fig:restricted}
    \label{}
\end{figure}
The third stage is now a normal \MCRAP\ search enriched with two pruning rules. Let $k$ be the current round.
\begin{itemize}
    \item  Labels $L$ are not merged into to $B_k(p)$ at stop $p$ if $\arr(L) > \backdep(m-k,p)$,
    \item  Labels $L$ can be removed from $B_r$ at stop $p$ if $\arr(L) > \backdep(m-k + 1,p) + \transtime(p)$.
\end{itemize}
The summand $\transtime(p)$ in the second rule is required since it was factored into \mbox{$\backdep(m-k + 1,p)$} and no transfer is performed.
This third stage computes exactly $\jr$.

The definition we gave for $\jstate$ differs from the one given by \citet{doi:10.1137/1.9781611975499.5}.
The original authors assign each journey $J\in \jstate$ its \textit{anchor journey} $\sjanch \in  \janch$, that has the maximum number of trips smaller or equal to $\tr(J)$.
They then define the restricted Pareto-set $\jdelling$ to be 
\begin{equation*}
     \begin{array}{lrl}
         \jdelling  &:=& \{ J \in \jstate \, | \, \text{ its anch. journ. } \sjanch \in \janch \text{ has } \arr(J) \leq \arr(\sjanch) + \asig \text{ and } \\
                    && \multicolumn{1}{r}{\tr(J) \leq \tr(\sjanch) + \tsig  \}} \\
                    &\subseteq& \{ J \in \jstate \,  |\, \exists  \sjanch \in \janch \text{ such that }  \arr(J) \leq \arr(\sjanch) + \asig \text{ and } \\
                    && \multicolumn{1}{r}{\tr(\sjanch) \leq \tr(J) \leq \tr(\sjanch) + \tsig  \}} \\
                    & \subseteq&  \jr.
    \end{array}
\end{equation*}

Note that  \TightBMRAP{} as presented here and in \citet{doi:10.1137/1.9781611975499.5} does, in fact, calculate $\jr$ and not $\jdelling$  and that $\jr$ is a more interesting set to compute as it does not seem beneficial to impose 
lower bounds on the objective function.
The difference between $\jr$ and $\jdelling$ is visualized in Figure \ref{fig:restricted}.
The first fact can be seen easily by considering the case where $\asig = \infty$. 
Then, the pruning scheme does not prune anything but all journeys with more than $m$ trips.
Specifically, no lower bound on the number of trips is applied.

\subsection{Speed-Up Techniques}\label{sec:speedups}
\paragraph{Price-Based Target Pruning}\label{sec:tp}
In \RAP\ as well as Dijkstra's algorithm, it is possible to use \textit{target pruning} \citep{DBLP:journals/transci/DellingPW15} to delete labels that are worse than the labels that have already been found at the target stop.
Naturally, the same speed-up technique is also possible for our algorithm.
Moreover, we need not use $\leqc$ to compare fare states.
Since the labels at the target stop are never updated and the price function $\price$ is non-decreasing, a partial journey already more expensive than the incumbent cheapest journey cannot be price-optimal.
Hence, in round $k$ of \MCRAP{}, we can prune all labels with a fare state $\state$ with $ \price(\tf(\state)) \geq \price^*$, with $\price^*$ being the best price at the target stop with at most $k$  trips.
We refer to this technique as \textit{Price-Based Target Pruning (PTP)} .
\paragraph{Fare-Specific Speed-ups}\label{sec:fss}
Certain dimensions in $\pmonoid$ might only be relevant for some tickets in $\tickets$.
For example, many short-distance tickets depend on the number of stops visited while this number is irrelevant for all other tickets that can be reached from that ticket. 
We can therefore alter the comparison operator $\leqc$ for those tickets to ignore the number of stops.
Hence, more labels become comparable, which results in a  smaller Pareto-set $\fssset$ with $\jprice \subseteq \fssset \subseteq \jstate$. 
When using \TightBMRAP{}, this results in a set $\rfssset$ with $\jrprice \subseteq \rfssset \subseteq \jrstate$.
We refer to this technique as \textit{fare-specific speed-up (FSS)}.

\section{Computational Results}\label{sec:results}
We implemented the \MCRAP\ algorithm in C++17 compiled with gcc 9.3.0 and $-\text{O}3$ optimization. 
All tests were conducted on Dell Poweredge M620 machines with 64 GB of RAM.
While the general structure of the MDV fare structure is captured in our model, our computations deviate from the prices charged by MDV in the following two cases:
A list of relations, that are, contrary to the general rules, not eligible for the short-distance discount, is considered.
Moreover, stops and fare zones that a route passes through without stopping are not represented in the available data and therefore cannot be considered.

Our dataset was built from  the publicly available timetable data \citep{MDVData2019} and fare structure of MDV \citep{MDVFare2019}.
Structured fare data is not public and was obtained separately via InfraDialog GmbH.
We extracted a timetable spanning two days from July 1, 2019 to July 2, 2019. 
The resulting timetable contains 4371 stops, 36670 trips, 5576 routes, and 845 footpaths. This original footpath set was not transitively closed.
Since \RAP\ requires a transitively closed footpath set \citep{DBLP:journals/transci/DellingPW15}, we computed its transitive closure and obtained 1029 footpaths.
We then chose a test set of 5000 origin-destination pairs (OD pairs) uniformly at random from the set of stops.
After removing all OD pairs that were not connected in the time interval starting at 08:00 a.m.\ on July 1, 2019, a total of 
4964 OD pairs remained.

As of 2019, the fare structure of MDV contained 56 fare zones, 17 cities with a city fare, and 30 overlap areas containing 191 stops. Overlap areas were implemented by route duplication, as lined out in Section \ref{par:neutral}.
After route duplication, the timetable contained 49072 trips in 7835 routes.
All queries were performed with a starting time of 08:00 a.m.

\begin{table}[t] 
{
\centering
\tiny 
\tabcolsep=0.07cm
\renewcommand{\arraystretch}{2}
\begin{tabular}{lccccccccccccccccccccccc}
\toprule
& \multicolumn{4}{c}{\fns Criteria}&\multicolumn{2}{c}{\fns Speed-Up}&\multicolumn{2}{c}{\fns Slack}
&\multicolumn{2}{c}{\fns \#Scan}  &\multicolumn{2}{c}{\fns Time[ms]} &\multicolumn{2}{c}{\fns \#Rounds} &\multicolumn{2}{c}{\fns \#Jn.} &\multicolumn{2}{c}{\fns \#PJn.} \\
\cmidrule(lr){2-5}\cmidrule(lr){6-7}\cmidrule(lr){8-9}\cmidrule(lr){10-11}\cmidrule(lr){12-13}\cmidrule(lr){14-15}\cmidrule(lr){16-17}\cmidrule(lr){18-19}
& \rotatebox[origin=c]{75}{\fns trips} & \rotatebox[origin=c]{75}{\fns time} & \rotatebox[origin=c]{75}{\fns zones} &
\rotatebox[origin=c]{75}{\fns fare} &  \rotatebox[origin=c]{75}{\fns PTP} & \rotatebox[origin=c]{75}{\fns FSS} & \rotatebox[origin=c]{75}{\fns arr} &
\rotatebox[origin=c]{75}{\fns trip} & \fns Avg. & \fns Sd.  & \fns Avg. & \fns Sd. & \fns Avg. & \fns Sd. & \fns Avg. & \fns Sd. & \fns Avg. & \fns Sd. \\
\hline
\fns  \RAP & \fns\textbullet& \fns\textbullet& \fns\textobullet& \fns\textobullet& \fns\textobullet& \fns\textobullet & -- & -- &
\fns 17295&\fns  5883&\fns 3.27&\fns1.31&\fns  6.64&\fns0.79&\fns1.53&\fns0.66 & \fns -- &  \fns  -- \\
\fns  \MCRAP & \fns\textbullet& \fns\textbullet& \fns\textbullet& \fns\textobullet& \fns\textobullet& \fns\textobullet & -- & -- &
\fns 29881 &\fns 6118 &\fns 4675&\fns 6172 &\fns6.92&\fns0.48&\fns 47.00 &\fns 45.95 &\fns  --  &\fns  -- \\
\fns  \MCRAP & \fns\textbullet& \fns\textbullet& \fns\textobullet& \fns\textbullet& \fns\textobullet& \fns\textbullet  & -- & --    & 
\fns 29891    &\fns    5025&\fns   957.54&\fns    260.73&\fns  6.97 &\fns        0.26&\fns10.05&\fns        13.27&\fns  2.62&\fns1.43 \\
\fns  \MCRAP & \fns\textbullet& \fns\textbullet& \fns\textobullet& \fns\textbullet& \fns\textbullet& \fns\textbullet  & -- & --  &   
\fns23460&        \fns8263 &  \fns243.09&         \fns283.08&\fns6.86&       \fns0.55 &\fns3.00&        \fns1.80  &\fns2.62&        \fns 1.43 \\
\hline
\fns  \TargetBMRAP & \fns\textbullet& \fns\textbullet& \fns\textobullet& \fns\textbullet& \fns\textbullet& \fns\textbullet  &\fns 15 & -- &
\fns 20132&\fns        8502  &\fns 71.16&\fns        87.10  &\fns 6.72&\fns       0.72  &\fns 1.72&\fns       0.82   &\fns 1.66&\fns       0.74 \\
\fns  \TargetBMRAP & \fns\textbullet& \fns\textbullet& \fns\textobullet& \fns\textbullet& \fns\textbullet& \fns\textbullet  &\fns 30 & -- &
\fns20637&\fns8412 &\fns 74.86&\fns88.75  &\fns 6.75&\fns0.70 &\fns   1.80&\fns0.89&\fns 1.72&\fns0.80 \\
\fns  \TargetBMRAP & \fns\textbullet& \fns\textbullet& \fns\textobullet& \fns\textbullet& \fns\textbullet& \fns\textbullet  &\fns 60 & -- & 
\fns 21394&\fns        8311 &\fns 81.45&\fns        91.57 &\fns 6.78&\fns       0.66  &\fns 1.95&\fns        1.00 &\fns1.85&\fns0.89 \\
\hline
\fns  \TightBMRAP & \fns\textbullet& \fns\textbullet& \fns\textobullet& \fns\textbullet& \fns\textobullet& \fns\textobullet  &\fns 15 &\fns 1 &
\fns3192&\fns        1963&\fns 8.44&\fns        14.46&\fns 5.17&\fns          1.08 &\fns 7.52&\fns        10.28&\fns 1.63&\fns       0.72\\
\fns  \TightBMRAP & \fns\textbullet& \fns\textbullet& \fns\textobullet& \fns\textbullet& \fns\textobullet& \fns\textobullet  &\fns 30 &\fns 1 &
\fns3551&\fns        2071&\fns 9.40&\fns        16.47&\fns 5.17&\fns          1.08&\fns9.06&\fns        12.29&\fns 1.68&\fns       0.78 \\
\fns  \TightBMRAP & \fns\textbullet& \fns\textbullet& \fns\textobullet& \fns\textbullet& \fns\textobullet& \fns\textobullet  &\fns 60 &\fns 1 &
\fns4316&\fns        2384&\fns 11.97&\fns          18.17&\fns 5.17&\fns          1.08&\fns13.29&\fns        16.02&\fns 1.78&\fns       0.86 \\
\fns  \TightBMRAP & \fns\textbullet& \fns\textbullet& \fns\textobullet& \fns\textbullet& \fns\textobullet& \fns\textobullet  &\fns 15 &\fns 2 &
\fns4606&\fns        3167&\fns 20.79&\fns        99.60&\fns 5.84&\fns        1.07&\fns11.90&\fns        21.05&\fns 1.66&\fns       0.74 \\
\fns  \TightBMRAP & \fns\textbullet& \fns\textbullet& \fns\textobullet& \fns\textbullet& \fns\textobullet& \fns\textobullet  &\fns 30 &\fns 2 &
\fns 5147&\fns        3154&\fns 22.75&\fns    104.12 &\fns 5.90&\fns  1.03&\fns14.09&\fns 23.87&\fns 1.71&\fns       0.80 \\
\fns  \TightBMRAP & \fns\textbullet& \fns\textbullet& \fns\textobullet& \fns\textbullet& \fns\textobullet& \fns\textobullet  &\fns 60 &\fns 2 &
\fns6487&\fns        3359&\fns 29.47&\fns         107.53&\fns 5.97&\fns       0.99&\fns20.98&\fns        29.70&\fns 1.83&\fns       0.88\\
\fns  \TightBMRAP & \fns\textbullet& \fns\textbullet& \fns\textobullet& \fns\textbullet& \fns\textbullet& \fns\textbullet  &\fns 15 &\fns 1  &
\fns2954&\fns         1828&\fns 6.22&\fns        3.14&\fns  5.17&\fns    1.08&\fns 1.68&\fns       0.79&\fns 1.63&\fns       0.72 \\
\fns  \TightBMRAP & \fns\textbullet& \fns\textbullet& \fns\textobullet& \fns\textbullet& \fns\textbullet& \fns\textbullet  &\fns 30 &\fns 1  &
\fns3274&\fns        1919&\fns 6.69&\fns           3.29&\fns  5.17&\fns          1.08&\fns 1.75&\fns       0.85&\fns 1.69&\fns       0.78 \\
\fns  \TightBMRAP & \fns\textbullet& \fns\textbullet& \fns\textobullet& \fns\textbullet& \fns\textbullet& \fns\textbullet  &\fns 60 &\fns 1  &
\fns3984&\fns        2209&\fns 7.82&\fns           3.67&\fns  5.17&\fns          1.08&\fns1.88&\fns       0.97&\fns 1.78&\fns       0.86 \\
\fns  \TightBMRAP & \fns\textbullet& \fns\textbullet& \fns\textobullet& \fns\textbullet& \fns\textbullet& \fns\textbullet  &\fns 15 &\fns 2 &
\fns4155&\fns        2905&\fns 8.11&\fns        6.28&\fns 5.47&\fns        1.08&\fns 1.71&\fns       0.81&\fns 1.66&\fns       0.74\\
\fns  \TightBMRAP & \fns\textbullet& \fns\textbullet& \fns\textobullet& \fns\textbullet& \fns\textbullet& \fns\textbullet  &\fns 30 &\fns 2 &
\fns4606&\fns         2888&\fns 8.76&\fns        6.32&\fns 5.53&\fns         1.0&\fns 1.78&\fns       0.88&\fns 1.71&\fns       0.80\\
\fns  \TightBMRAP & \fns\textbullet& \fns\textbullet& \fns\textobullet& \fns\textbullet& \fns\textbullet& \fns\textbullet  &\fns 60 &\fns 2 &
\fns5758&\fns        3055&\fns 10.68&\fns        6.73&\fns 5.62&\fns        1.05&\fns 1.93&\fns        1.00&\fns 1.83&\fns       0.88\\
\bottomrule
\end{tabular}}

\caption{Computational Results.
Evaluation of different \RAP{} variants on the MDV dataset. All experiments were conducted with a maximum of seven rounds.
 For each algorithm, the table reports the optimization criteria, the speed-up techniques employed, and the arrival and trip slacks, if applicable.
 We report the number of scanned routes (\textit{\#Scan}), the average running time (\textit{Time}), the number of rounds performed (\textit{\#Rounds}),
 the number of journeys found (\textit{\#Jn.}) and the number of journeys that are dominated w.r.t. to the price (\textit{\#PJn.}). For each result, both the average and standard deviation are reported.
 The algorithms in rows 2-17 run at least one \RAP{} and exactly one \MCRAP query.
 In this case, \textit{\#Scan} and \textit{Time} are only given for the \MCRAP{} run, while the running time is summed up overall \RAP{} and \MCRAP{} invocations.}\label{table:seven}
\end{table}

In a first experiment, we ran several algorithms for each OD pair: the standard \RAP-algorithm; a \MCRAP{} variant optimizing for the arrival time, the number of trips, and the set of fare zones;
then, optimizing for the arrival time, number of trips, and the fare state, two \MCRAP{} variants as well as \TargetBMRAP{} and \TightBMRAP{}  with various configurations for arrival and trip slacks. 
This results in a total of 19 algorithms in the first experiment.
When optimizing for fare state, a postprocessing step is performed to remove all journeys that are in the Pareto-set with regard to fare state but not to price.
Note that the \MCRAP{} variant for fare zones is a relaxed version of the purely monoid-based approach from Section \ref{subsec:modelmonoids} and provides a lower bound on its performance. The full approach was implemented but ran out of memory on most queries and is therefore omitted in the results.

The experiment was conducted with a maximum of seven rounds. This number was chosen, as we believe it represents a sufficiently generous upper bound on the maximum number of transfers a traveler is willing to undertake. The results of the experiment are reported in Table \ref{table:seven}.
A standard \RAP\ run takes on average 3.27 ms with a standard deviation of 1.31 ms.
The computed Pareto-set contains 1.53 journeys on average.
If only fare zones are considered as additional operation criterion, the average runtime increases 
to 4.67\,s, with a standard deviation of 6.17\,s.
This indicates that run times of up to around 10\,s are not out of the ordinary. 

While FSS in \MCRAP\ alone does not suffice to obtain acceptable run times, the combination of FSS and PTP 
produces an average run time of  243.09 ms.
Although this performance is no longer prohibitive for practical application, the standard deviation remains high at 283 ms. 
\begin{table}[t]
\centering
\tiny 
\tabcolsep=0.07cm
\renewcommand{\arraystretch}{2}
{\begin{tabular}{lcccccccccccccccccc}
\toprule
& \multicolumn{2}{c}{\footnotesize Slack} &\multicolumn{2}{c}{\footnotesize \#Scan}  &\multicolumn{2}{c}{\footnotesize Time[ms]} &\multicolumn{3}{c}{\footnotesize \#Rounds} &\multicolumn{2}{c}{\footnotesize \#Jn.} &\multicolumn{2}{c}{\footnotesize \#PJn.} \\
\cmidrule(lr){2-3}\cmidrule(lr){4-5}\cmidrule(lr){6-7}\cmidrule(lr){8-10}\cmidrule(lr){11-12}\cmidrule(lr){13-14}
&  \rotatebox[origin=c]{75}{\footnotesize arr} &
        \rotatebox[origin=c]{75}{\footnotesize trip} & \footnotesize Avg. & \footnotesize Sd.  &  \footnotesize Avg. & \footnotesize Sd. & \footnotesize Avg. & \footnotesize Sd. &\footnotesize Max & \footnotesize Avg. & \footnotesize Sd. & \footnotesize Avg. & \footnotesize Sd. \\
\hline
        \footnotesize  \RAP & -- & --&
\fns17496&\fns        6111&\fns 3.28 &\fns         1331.9&\fns 7.31&\fns         1.35&\fns 12 &\fns 1.54&\fns       0.66&\fns -- &\fns       -- \\
\footnotesize  \MCRAP & -- & --  &   
\fns32757&\fns        16429&\fns 389.52&\fns         710.75&\fns 11.10&\fns        3.24&\fns 21 &\fns  3.16&\fns        1.91&\fns 2.77&\fns        1.56\\
\hline
\footnotesize  \TargetBMRAP & \fns 15  & -- &
 \fns23007&\fns        12106&\fns 82.47&\fns         125.89&\fns 8.57&\fns        2.34&\fns 17 &\fns 1.73&\fns       0.84&\fns 1.68&\fns       0.77\\
 \footnotesize  \TargetBMRAP & \fns 30  & --&
\fns23724&\fns        12093&\fns 86.38&\fns         128.22&\fns 8.73&\fns        2.33&\fns 17&\fns 1.81&\fns       0.91&\fns 1.73&\fns       0.82\\
\footnotesize  \TargetBMRAP & \fns 60  & -- & 
\fns24980&\fns        12186&\fns 94.53&\fns         131.83&\fns 9.00&\fns        2.35&\fns 17&\fns 1.97&\fns        1.03&\fns 1.87&\fns       0.91\\
\hline
\footnotesize  \TightBMRAP & \fns 15 &\fns 1  &
\fns2985&\fns        1871&\fns 6.68&\fns           3.28&\fns 5.20&\fns        1.14&\fns 9 &\fns 1.69&\fns       0.81&\fns 1.63&\fns       0.73\\
\footnotesize  \TightBMRAP & \fns 30 &\fns 1 &
\fns3309&\fns        1973&\fns 7.19&\fns        3.44&\fns 5.20&\fns        1.14&\fns 9 &\fns 1.75&\fns        0.87&\fns 1.69&\fns        0.79\\
\footnotesize  \TightBMRAP & \fns 60 &\fns 1  &
\fns4039&\fns        2307&\fns 8.36&\fns        3.89&\fns 5.20&\fns        1.14&\fns 9 &\fns 1.89&\fns       0.96&\fns 1.79&\fns       0.89\\
\footnotesize  \TightBMRAP & \fns 15 &\fns 2 &
\fns4360&\fns        3145&\fns 9.02 &\fns        7.27&\fns 5.52&\fns        1.16&\fns 10 &\fns 1.72&\fns       0.83&\fns 1.66&\fns       0.76\\
 \footnotesize  \TightBMRAP & \fns 30 &\fns 2 &
\fns4833&\fns        3151&\fns 9.68&\fns        7.43&\fns 5.58&\fns         1.15&\fns 10 &\fns 1.79&\fns        0.90&\fns 1.72&\fns       0.81\\
\footnotesize  \TightBMRAP & \fns 60 &\fns 2 &
\fns6071&\fns        3378&\fns 11.87&\fns         7.97&\fns 5.67&\fns        1.16&\fns 10 &\fns 1.94&\fns        1.02&\fns 1.84&\fns       0.90\\
\bottomrule
\end{tabular}}
\caption{Computational Results.
Evaluation of different \RAP{} variants on the MDV dataset. All experiments were conducted with a maximum of 25 rounds. The maximum number of rounds performed across all algorithms was 21. Hence, no query was cancelled prematurely. 
All algorithms other than \RAP{} optimized for arrival time, number of trips and fare state. For all of those, both FFS and PTP were activated. 
The same performance indicators as in Table \ref{table:seven} are reported.
For \TargetBMRAP{} and \TightBMRAP{}, \textit{\#Scan} and {\#Rounds} report only on the last \MCRAP{} call whereas \textit{Time} reports the overall run time.
}\label{table:twentyfive}
\end{table}

The Pareto-set computed with \MCRAP{} contains 3 journeys on average, of which 2.62 are also price-optimal.
Turning off PTP increases the number of computed journeys to  10.5; the Pareto-set of the zone-based \MCRAP{} contains 47 journeys on average.
Hence, even though all tickets of MDV are in the full-comparability set $\fullcomp$, a high number of superfluous journeys will be generated when no additional techniques are employed.
This effect is mainly caused by the fare zones, as they form an only partially ordered set.
The 2.62 price-optimal journeys mark an increase of 71\% over the 1.53 journeys found with \RAP.
It might, however, contain journeys with an undesirable trade-off between arrival time and number of trips and price.
To obtain restricted Pareto-sets with a reasonable trade-off, we ran \TargetBMRAP{} and \TightBMRAP{} with arrival time slacks of 
15 min, 30 min and 60 min and in the case of \TightBMRAP{} with trip slacks of 1 or 2.
Using \TargetBMRAP{} reduces run times to between 71.16 ms and 81.45 ms while restricting the size of the Pareto-set to between 1.66 and 1.85.
This corresponds to between 8.5\% and 21\% more journeys compared to \RAP.
We ran \TightBMRAP{}  both with and without PTP and FSS. \TightBMRAP{} performs reasonably well even without PTP and FSS on average with run times of up to 29.47 ms.
However, a  comparatively high standard deviation of up to 107.53 ms hints at high performance variability.
When using both PTP and FSS run times decrease to at most 10.68 ms. Even more pronounced is the decrease in the standard deviation to at most 6.73 ms.
Compared to \RAP{} there are between 6.5\% and 19.6\% more journeys.
Consequently, conditional fare networks used within \TightBMRAP{} appear well-suited to provide the user with price-optimized alternative routes while 
increasing run times only insignificantly.

A second experiment was conducted without an upper bound on the number of rounds. 
The table $\arr$ was implemented as a fixed-size array with space for 25 rounds.
Since the maximum number of rounds performed was 21, no journeys  were cut off due to early termination. 
We excluded all algorithms that had already performed poorly in the first experiment.
Namely, these are the \TightBMRAP{} variants without additional speed-up techniques, \MCRAP{} for fare zones and \MCRAP{} for fare states without price-based target-pruning.

For \MCRAP, we see a significant increase in the run time of about 60\% and a even more pronounced increase in its standard deviation of 151\% compared to the variant with only seven rounds.
All variants of \TargetBMRAP{} exhibit  similar behavior, albeit to a lesser degree.
Here, the average run time increased by about 16\%  and the standard deviation by about 44\%.
Note that for \RAP\ the run time increased by only a marginal 0.01 ms and that, when compared to \RAP{}, \MCRAP{} needs to perform significantly more rounds.

For all settings of slack variables for \TightBMRAP{}, the increase of the run time remains minimal at around a millisecond. Hence, \TightBMRAP{} remains highly competitive, whereas
the simple \MCRAP{} implementation suffers from considerably degraded performance.
It is furthermore noteworthy that in all slack settings, the multi-criteria part of \TightBMRAP{} needs to both scan significantly fewer routes and  perform fewer rounds than even the standard \RAP.
This clearly speaks to the strength of the pruning scheme used in \TightBMRAP.

\section{Conclusion}\label{sec:conclusion}
We presented conditional fare networks, a novel framework for modeling complex fare structures of public transportation provi\-ders.  It is independent of the MOSP algorithm used and can be used to solve price-optimal earliest arrival queries in real-world networks.  
Since fare structures are often composed of various fare strategies, this requires the optimization of several objective functions.
In the MDV case study, these were the fare zones, city fares, transfers, the number of stops visited, and the length of the path in kilometers.
Performing naive multi-objective queries for all these objectives results in high run times with high variance and the computation of many journeys that are not price-optimal.
In contrast, using a CFN-based variant of the \MCRAP{} algorithm, we were able to mitigate these effects and reduce run times to around 400 ms on average, which we deem acceptable for commercial applications.
Combining CFNs with the \TightBMRAP{} algorithm reduced run times further to at most 12 ms with low variance, while still computing a reasonably sized restricted Pareto-set when choosing appropriate arrival and trip slacks.
Fare structures can differ quite significantly between public transportation providers.
Hence, a systematic evaluation of CFNs on other public transit networks is certainly worthwhile.
As MDV operates in a largely rural area with two only medium-sized urban centers, a study of larger urban centers such as Berlin or Madrid seems especially interesting. 
However, while timetables are widely available, fare data is not.
Especially, machine-readable mappings from stations to fare zones are generally not publicly available.
When they are, the data is often incomplete and requires a significant manual polishing effort.

\section*{Acknowledgement}
We thank MDV and InfraDialog Gmbh for providing the data for this study. 
We owe special gratitude to our master's student Rick Grap for implementing the \TightBMRAP-algorithm and pointing out the difference between the sets $\jdelling$ and $\jr$ 
in Section \ref{sec:restricted}.

\bibliographystyle{elsarticle-num-names}
\bibliography{EulLinBornPriceOpt} 

\begin{thebibliography}{36}
\expandafter\ifx\csname natexlab\endcsname\relax\def\natexlab#1{#1}\fi
\providecommand{\url}[1]{\texttt{#1}}
\providecommand{\href}[2]{#2}
\providecommand{\path}[1]{#1}
\providecommand{\DOIprefix}{doi:}
\providecommand{\ArXivprefix}{arXiv:}
\providecommand{\URLprefix}{URL: }
\providecommand{\Pubmedprefix}{pmid:}
\providecommand{\doi}[1]{\href{http://dx.doi.org/#1}{\path{#1}}}
\providecommand{\Pubmed}[1]{\href{pmid:#1}{\path{#1}}}
\providecommand{\bibinfo}[2]{#2}
\ifx\xfnm\relax \def\xfnm[#1]{\unskip,\space#1}\fi
\bibitem[{Brough et~al.(2022)Brough, Freedman, and Phillips}]{Brough2022}
\bibinfo{author}{R.~Brough}, \bibinfo{author}{M.~Freedman},
  \bibinfo{author}{D.~C. Phillips},
\newblock \bibinfo{title}{Experimental evidence on the effects of means-tested
  public transportation subsidies on travel behavior},
\newblock \bibinfo{journal}{Regional Science and Urban Economics}
  \bibinfo{volume}{96} (\bibinfo{year}{2022}) \bibinfo{pages}{103803}.
  \URLprefix
  \url{https://www.sciencedirect.com/science/article/pii/S0166046222000436}.
  \DOIprefix\doi{https://doi.org/10.1016/j.regsciurbeco.2022.103803}.
\bibitem[{Bull et~al.(2021)Bull, Muñoz, and Silva}]{Bull2021}
\bibinfo{author}{O.~Bull}, \bibinfo{author}{J.~C. Muñoz},
  \bibinfo{author}{H.~E. Silva},
\newblock \bibinfo{title}{The impact of fare-free public transport on travel
  behavior: Evidence from a randomized controlled trial},
\newblock \bibinfo{journal}{Regional Science and Urban Economics}
  \bibinfo{volume}{86} (\bibinfo{year}{2021}) \bibinfo{pages}{103616}.
  \URLprefix
  \url{https://www.sciencedirect.com/science/article/pii/S016604622030301X}.
  \DOIprefix\doi{https://doi.org/10.1016/j.regsciurbeco.2020.103616}.
\bibitem[{Chen et~al.(2020)Chen, Ma, and Bai}]{Chen}
\bibinfo{author}{X.~Chen}, \bibinfo{author}{J.~Ma}, \bibinfo{author}{X.~Bai},
  \bibinfo{title}{Mode Choice Behavior Analysis under the Impact of Transfer
  Fare Discount: A Case Study from Beijing Public Transit System},
  \bibinfo{year}{2020}, pp. \bibinfo{pages}{290--299}. \URLprefix
  \url{https://ascelibrary.org/doi/abs/10.1061/9780784482902.033}.
  \DOIprefix\doi{10.1061/9780784482902.033}.
  \href{http://arxiv.org/abs/https://ascelibrary.org/doi/pdf/10.1061/9780784482902.033}{{\tt
  arXiv:https://ascelibrary.org/doi/pdf/10.1061/9780784482902.033}}.
\bibitem[{Blumenberg and Agrawal(2014)}]{Blumenberg2014}
\bibinfo{author}{E.~Blumenberg}, \bibinfo{author}{A.~W. Agrawal},
\newblock \bibinfo{title}{Getting around when you’re just getting by:
  Transportation survival strategies of the poor},
\newblock \bibinfo{journal}{Journal of Poverty} \bibinfo{volume}{18}
  (\bibinfo{year}{2014}) \bibinfo{pages}{355--378}. \URLprefix
  \url{https://doi.org/10.1080/10875549.2014.951905}.
  \DOIprefix\doi{10.1080/10875549.2014.951905}.
  \href{http://arxiv.org/abs/https://doi.org/10.1080/10875549.2014.951905}{{\tt
  arXiv:https://doi.org/10.1080/10875549.2014.951905}}.
\bibitem[{Rosenblum(2020)}]{Rosenblum2020}
\bibinfo{author}{J.~Rosenblum}, \bibinfo{title}{Expanding Access to the City:
  How Public Transit Fare Policy Shapes Travel Decision Making and Behavior of
  Low-Income riders}, Ph.D. thesis, Massachusetts Institute of Technology.
  Department of Urban Studies and Planning, \bibinfo{year}{2020}. \URLprefix
  \url{https://hdl.handle.net/1721.1/127617}.
\bibitem[{Fleishman et~al.(1996)Fleishman, Shaw, Joshi, Freeze, and
  Oram}]{fleishman1996fare}
\bibinfo{author}{D.~Fleishman}, \bibinfo{author}{N.~Shaw},
  \bibinfo{author}{A.~Joshi}, \bibinfo{author}{R.~Freeze},
  \bibinfo{author}{R.~Oram},
\newblock \bibinfo{title}{Fare policies, structures and technologies, tcrp
  report 10},
\newblock \bibinfo{journal}{Transport Cooperative Research Program,
  Transportation Research Board, Washington DC}  (\bibinfo{year}{1996}).
\bibitem[{Blanco et~al.(2016)Blanco, Bornd{\"o}rfer, Hoang, Kaier, Schlechte,
  and Schlobach}]{BlancoBorndoerferHoangetal.2016}
\bibinfo{author}{M.~Blanco}, \bibinfo{author}{R.~Bornd{\"o}rfer},
  \bibinfo{author}{N.-D. Hoang}, \bibinfo{author}{A.~Kaier},
  \bibinfo{author}{T.~Schlechte}, \bibinfo{author}{S.~Schlobach},
  \bibinfo{title}{The Shortest Path Problem with Crossing Costs},
  \bibinfo{type}{Technical Report} \bibinfo{number}{16-70}, ZIB,
  \bibinfo{address}{Takustr. 7, 14195 Berlin}, \bibinfo{year}{2016}. \URLprefix
  \url{urn:nbn:de:0297-zib-61240}.
\bibitem[{Delling et~al.(2015)Delling, Pajor, and
  Werneck}]{DBLP:journals/transci/DellingPW15}
\bibinfo{author}{D.~Delling}, \bibinfo{author}{T.~Pajor},
  \bibinfo{author}{R.~F. Werneck},
\newblock \bibinfo{title}{Round-based public transit routing},
\newblock \bibinfo{journal}{Transportation Science} \bibinfo{volume}{49}
  (\bibinfo{year}{2015}) \bibinfo{pages}{591--604}. \URLprefix
  \url{https://doi.org/10.1287/trsc.2014.0534}.
  \DOIprefix\doi{10.1287/trsc.2014.0534}.
\bibitem[{Delling et~al.(2019)Delling, Dibbelt, and
  Pajor}]{doi:10.1137/1.9781611975499.5}
\bibinfo{author}{D.~Delling}, \bibinfo{author}{J.~Dibbelt},
  \bibinfo{author}{T.~Pajor},
\newblock \bibinfo{title}{Fast and exact public transit routing with restricted
  pareto sets},
\newblock in: \bibinfo{booktitle}{2019 Proceedings of the Twenty-First Workshop
  on Algorithm Engineering and Experiments (ALENEX)}, \bibinfo{year}{2019}, pp.
  \bibinfo{pages}{54--65}. \URLprefix
  \url{https://epubs.siam.org/doi/abs/10.1137/1.9781611975499.5}.
  \DOIprefix\doi{10.1137/1.9781611975499.5}.
  \href{http://arxiv.org/abs/https://epubs.siam.org/doi/pdf/10.1137/1.9781611975499.5}{{\tt
  arXiv:https://epubs.siam.org/doi/pdf/10.1137/1.9781611975499.5}}.
\bibitem[{Bast et~al.(2016)Bast, Delling, Goldberg, M{\"u}ller-Hannemann,
  Pajor, Sanders, Wagner, and Werneck}]{Bast2016RoutePI}
\bibinfo{author}{H.~Bast}, \bibinfo{author}{D.~Delling}, \bibinfo{author}{A.~V.
  Goldberg}, \bibinfo{author}{M.~M{\"u}ller-Hannemann},
  \bibinfo{author}{T.~Pajor}, \bibinfo{author}{P.~Sanders},
  \bibinfo{author}{D.~Wagner}, \bibinfo{author}{R.~F. Werneck},
\newblock \bibinfo{title}{Route planning in transportation networks},
\newblock in: \bibinfo{booktitle}{Algorithm Engineering - Selected Results and
  Surveys}, volume \bibinfo{volume}{9220} of \textit{\bibinfo{series}{Lecture
  Notes in Computer Science}}, \bibinfo{publisher}{Springer},
  \bibinfo{year}{2016}, pp. \bibinfo{pages}{19--80}.
\bibitem[{Martins(1984)}]{martins1984multicriteria}
\bibinfo{author}{E.~Q.~V. Martins},
\newblock \bibinfo{title}{On a multicriteria shortest path problem},
\newblock \bibinfo{journal}{European Journal of Operational Research}
  \bibinfo{volume}{16} (\bibinfo{year}{1984}) \bibinfo{pages}{236--245}.
\bibitem[{de~las Casas et~al.(2021)de~las Casas, Sedeno-Noda, and
  Bornd{\"o}rfer}]{MaristanydelasCasasSedenoNodaBorndoerfer2021}
\bibinfo{author}{P.~M. de~las Casas}, \bibinfo{author}{A.~Sedeno-Noda},
  \bibinfo{author}{R.~Bornd{\"o}rfer},
\newblock \bibinfo{title}{An improved multiobjective shortest path algorithm},
\newblock \bibinfo{journal}{Computers \& Operations Research}
  \bibinfo{volume}{135} (\bibinfo{year}{2021}).
  \DOIprefix\doi{10.1016/j.cor.2021.105424}.
\bibitem[{Berger and M{\"u}ller-Hannemann(2009)}]{Berger_subpath-optimalityof}
\bibinfo{author}{A.~Berger}, \bibinfo{author}{M.~M{\"u}ller-Hannemann},
  \bibinfo{title}{{S}ubpath-{O}ptimality of {M}ulti-{C}riteria {S}hortest
  {P}aths in {T}ime- and {E}vent-{D}ependent {N}etworks},
  \bibinfo{type}{Technical Report}, Institute of Computer Science,
  Martin-Luther-Universit{\"a}t Halle-Wittenberg, \bibinfo{year}{2009}.
  \URLprefix
  \url{http://wcms.uzi.uni-halle.de/download.php?down=10850\&elem=2163494}.
\bibitem[{Disser et~al.(2008)Disser, M\"{u}ller-Hannemann, and
  Schnee}]{Disser2008MulticriteriaSP}
\bibinfo{author}{Y.~Disser}, \bibinfo{author}{M.~M\"{u}ller-Hannemann},
  \bibinfo{author}{M.~Schnee},
\newblock \bibinfo{title}{Multi-criteria shortest paths in time-dependent train
  networks},
\newblock in: \bibinfo{editor}{C.~C. McGeoch} (Ed.),
  \bibinfo{booktitle}{Proceedings of the 7th International Conference on
  Experimental Algorithms}, WEA'08, \bibinfo{publisher}{Springer-Verlag},
  \bibinfo{address}{Berlin, Heidelberg}, \bibinfo{year}{2008}, pp.
  \bibinfo{pages}{347--361}. \URLprefix
  \url{http://dl.acm.org/citation.cfm?id=1788888.1788914}.
  \DOIprefix\doi{https://doi.org/10.1007/978-3-540-68552-4_26}.
\bibitem[{M\"{u}ller-Hannemann and Schnee(2005)}]{HannemannSchneePayingLess}
\bibinfo{author}{M.~M\"{u}ller-Hannemann}, \bibinfo{author}{M.~Schnee},
\newblock \bibinfo{title}{Paying less for train connections with motis},
\newblock in: \bibinfo{booktitle}{Proceedings of the 5th Workshop on
  Algorithmic Methods and Models for Optimization of Railways},
  volume~\bibinfo{volume}{2} of \textit{\bibinfo{series}{OpenAccess Series in
  Informatics}}, \bibinfo{year}{2005}, p. \bibinfo{pages}{657}.
  \DOIprefix\doi{10.4230/OASIcs.ATMOS.2005.657}.
\bibitem[{Reinhardt and Pisinger(2011)}]{Reinhardt2011}
\bibinfo{author}{L.~B. Reinhardt}, \bibinfo{author}{D.~Pisinger},
\newblock \bibinfo{title}{Multi-objective and multi-constrained non-additive
  shortest path problems},
\newblock \bibinfo{journal}{Computers \& Operations Research}
  \bibinfo{volume}{38} (\bibinfo{year}{2011}) \bibinfo{pages}{605--616}.
  \URLprefix
  \url{https://www.sciencedirect.com/science/article/pii/S0305054810001656}.
  \DOIprefix\doi{https://doi.org/10.1016/j.cor.2010.08.003}.
\bibitem[{Schöbel and Urban(2021)}]{schoebel2021cheapest}
\bibinfo{author}{A.~Schöbel}, \bibinfo{author}{R.~Urban},
\newblock \bibinfo{title}{The cheapest ticket problem in public transport},
\newblock \bibinfo{journal}{pre-print}  (\bibinfo{year}{2021}).
  \DOIprefix\doi{10.48550/arXiv.2106.10521}, \bibinfo{note}{{a}rXiv:2106.10521
  [math.OC]}.
\bibitem[{Broersma et~al.(2005)Broersma, Li, Woeginger, and
  Zhang}]{broersma2005paths}
\bibinfo{author}{H.~Broersma}, \bibinfo{author}{X.~Li},
  \bibinfo{author}{G.~Woeginger}, \bibinfo{author}{S.~Zhang},
\newblock \bibinfo{title}{Paths and cycles in colored graphs.},
\newblock \bibinfo{journal}{The Australasian Journal of Combinatorics}
  \bibinfo{volume}{31} (\bibinfo{year}{2005}) \bibinfo{pages}{299--311}.
  \DOIprefix\doi{10.1145/62.2737}.
\bibitem[{Blanco et~al.(2017)Blanco, Bornd{\"o}rfer, Ho{\`a}ng, Kaier, Casas,
  Schlechte, and Schlobach}]{BlancoBorndoerferHoangetal.2017}
\bibinfo{author}{M.~Blanco}, \bibinfo{author}{R.~Bornd{\"o}rfer},
  \bibinfo{author}{N.~D. Ho{\`a}ng}, \bibinfo{author}{A.~Kaier},
  \bibinfo{author}{P.~M. Casas}, \bibinfo{author}{T.~Schlechte},
  \bibinfo{author}{S.~Schlobach},
\newblock \bibinfo{title}{{Cost Projection Methods for the Shortest Path
  Problem with Crossing Costs}},
\newblock in: \bibinfo{editor}{G.~D'Angelo}, \bibinfo{editor}{T.~Dollevoet}
  (Eds.), \bibinfo{booktitle}{17th Workshop on Algorithmic Approaches for
  Transportation Modelling, Optimization, and Systems (ATMOS 2017)},
  volume~\bibinfo{volume}{59} of \textit{\bibinfo{series}{OpenAccess Series in
  Informatics (OASIcs)}}, \bibinfo{publisher}{Schloss Dagstuhl--Leibniz-Zentrum
  fuer Informatik}, \bibinfo{address}{Dagstuhl, Germany}, \bibinfo{year}{2017},
  pp. \bibinfo{pages}{15:1--15:14}. \URLprefix
  \url{http://drops.dagstuhl.de/opus/volltexte/2017/7893}.
  \DOIprefix\doi{10.4230/OASIcs.ATMOS.2017.15}.
\bibitem[{G{\"u}ndling(2020)}]{GundlingDissertation2020}
\bibinfo{author}{F.~G{\"u}ndling}, \bibinfo{title}{Efficient Algorithms for
  Intermodal Routing and Monitoring in Travel Information Systems}, Ph.D.
  thesis, Technische Universit{\"a}t, \bibinfo{address}{Darmstadt},
  \bibinfo{year}{2020}. \URLprefix
  \url{http://tuprints.ulb.tu-darmstadt.de/14212/}.
  \DOIprefix\doi{https://doi.org/10.25534/tuprints-00014212}.
\bibitem[{Hopcroft and Ullman(1979)}]{Hopcroft+Ullman/79/Introduction}
\bibinfo{author}{J.~E. Hopcroft}, \bibinfo{author}{J.~D. Ullman},
  \bibinfo{title}{Introduction to Automata Theory, Languages, and Computation},
  \bibinfo{publisher}{Addison-Wesley Publishing Company}, \bibinfo{year}{1979}.
\bibitem[{Barrett et~al.(2000)Barrett, Jacob, and
  Marathe}]{Barrett:2000:FPP:586846.586970}
\bibinfo{author}{C.~Barrett}, \bibinfo{author}{R.~Jacob},
  \bibinfo{author}{M.~Marathe},
\newblock \bibinfo{title}{Formal-language-constrained path problems},
\newblock \bibinfo{journal}{SIAM J. Comput.} \bibinfo{volume}{30}
  (\bibinfo{year}{2000}) \bibinfo{pages}{809--837}. \URLprefix
  \url{https://doi.org/10.1137/S0097539798337716}.
  \DOIprefix\doi{10.1137/S0097539798337716}.
\bibitem[{Zimmermann(1981)}]{ZimmermannOrderedCombinatorialOptimization}
\bibinfo{author}{U.~Zimmermann}, \bibinfo{title}{Linear and combinatorial
  optimization in ordered algebraic structures}, volume~\bibinfo{volume}{10} of
  \textit{\bibinfo{series}{Annaly of discrete mathematics}},
  \bibinfo{publisher}{North-Holland}, \bibinfo{year}{1981}.
\bibitem[{Mohri(2002)}]{Mohri:2002:SFA:639508.639512}
\bibinfo{author}{M.~Mohri},
\newblock \bibinfo{title}{Semiring frameworks and algorithms for
  shortest-distance problems},
\newblock \bibinfo{journal}{J. Autom. Lang. Comb.} \bibinfo{volume}{7}
  (\bibinfo{year}{2002}) \bibinfo{pages}{321--350}. \URLprefix
  \url{http://dl.acm.org/citation.cfm?id=639508.639512}.
\bibitem[{Parmentier(2019)}]{Parmentier2019a}
\bibinfo{author}{A.~Parmentier},
\newblock \bibinfo{title}{Algorithms for non-linear and stochastic resource
  constrained shortest path},
\newblock \bibinfo{journal}{Mathematical Methods of Operations Research}
  \bibinfo{volume}{89} (\bibinfo{year}{2019}) \bibinfo{pages}{281--317}.
  \URLprefix \url{https://doi.org/10.1007/s00186-018-0649-x}.
  \DOIprefix\doi{10.1007/s00186-018-0649-x}.
\bibitem[{Euler and Bornd{\"o}rfer(2019)}]{euler_et_al:OASIcs:2019:11424}
\bibinfo{author}{R.~Euler}, \bibinfo{author}{R.~Bornd{\"o}rfer},
\newblock \bibinfo{title}{{A Graph- and Monoid-Based Framework for
  Price-Sensitive Routing in Local Public Transportation Networks}},
\newblock in: \bibinfo{editor}{V.~Cacchiani},
  \bibinfo{editor}{A.~Marchetti-Spaccamela} (Eds.), \bibinfo{booktitle}{19th
  Symposium on Algorithmic Approaches for Transportation Modelling,
  Optimization, and Systems (ATMOS 2019)}, volume~\bibinfo{volume}{75} of
  \textit{\bibinfo{series}{OpenAccess Series in Informatics (OASIcs)}},
  \bibinfo{publisher}{Schloss Dagstuhl--Leibniz-Zentrum fuer Informatik},
  \bibinfo{address}{Dagstuhl, Germany}, \bibinfo{year}{2019}, pp.
  \bibinfo{pages}{12:1--12:15}. \URLprefix
  \url{http://drops.dagstuhl.de/opus/volltexte/2019/11424}.
  \DOIprefix\doi{10.4230/OASIcs.ATMOS.2019.12}.
\bibitem[{Bornd{\"o}rfer et~al.(2018)Bornd{\"o}rfer, Euler, Karbstein, and
  Mett}]{BorndoerferEulerKarbsteinetal.2018}
\bibinfo{author}{R.~Bornd{\"o}rfer}, \bibinfo{author}{R.~Euler},
  \bibinfo{author}{M.~Karbstein}, \bibinfo{author}{F.~Mett},
  \bibinfo{title}{{E}in mathematisches {M}odell zur {B}eschreibung von
  {Preissystemen} im {\"o}{V}}, \bibinfo{type}{Technical Report}
  \bibinfo{number}{18-47}, ZIB, \bibinfo{address}{Takustr. 7, 14195 Berlin},
  \bibinfo{year}{2018}. \URLprefix \url{urn:nbn:de:0297-zib-70564}.
\bibitem[{Bornd{\"o}rfer et~al.(2021)Bornd{\"o}rfer, Euler, and
  Karbstein}]{BorndoerferEulerKarbstein2021}
\bibinfo{author}{R.~Bornd{\"o}rfer}, \bibinfo{author}{R.~Euler},
  \bibinfo{author}{M.~Karbstein},
\newblock \bibinfo{title}{Ein graphen-basiertes modell zur beschreibung von
  preissystemen im {\"o}ffentlichen nahverkehr},
\newblock \bibinfo{journal}{HEUREKA 21} \bibinfo{volume}{002/127}
  (\bibinfo{year}{2021}) \bibinfo{pages}{1 -- 15}. \URLprefix
  \url{https://verlag.fgsv-datenbanken.de/tagungsbaende?kat=HEUREKA\&subkat=FGSV+002\%2F127+\%282021\%29\&fanr=\&va=\&titel=\&text=\&autor=\&tagungsband=1256\&_titel=Ein+Graphen-basiertes+Modell+zur+Beschreibung+von+Preissystemen+im+\%C3\%B6ffentlichen+Nahverkehr}.
\bibitem[{{Mitteldeutscher Verkehrsverbund GmbH}(2019)}]{MDVFare2019}
\bibinfo{author}{{Mitteldeutscher Verkehrsverbund GmbH}}, \bibinfo{title}{{MDV}
  fares},
  \bibinfo{howpublished}{\url{https://www.mdv.de/tickets/befoerderungsbedingungen-tarifbestimmungen/}},
  \bibinfo{year}{2019}. \bibinfo{note}{Accessed: 2019-08-11}.
\bibitem[{{Verkehrsverbund Bremen/Niedersachsen GmbH}(2019)}]{VBN2019}
\bibinfo{author}{{Verkehrsverbund Bremen/Niedersachsen GmbH}},
  \bibinfo{title}{{VBN} nightliner fares},
  \bibinfo{howpublished}{\url{https://web.archive.org/web/20200930070236///https://www.vbn.de/tickets/ticketangebot/nachtlinienzuschlag/}},
  \bibinfo{year}{2019}. \bibinfo{note}{Accessed: 2023-11-10}.
\bibitem[{{Verkehrsverbund Bremen/Niedersachsen GmbH}(2023)}]{VBB2023Fares}
\bibinfo{author}{{Verkehrsverbund Bremen/Niedersachsen GmbH}},
  \bibinfo{title}{{VBN} fares},
  \bibinfo{howpublished}{\url{https://www.vbn.de/tickets/tarifbestimmungen}},
  \bibinfo{year}{2023}. \bibinfo{note}{Accessed: 2023-11-10}.
\bibitem[{{Mitteldeutscher Verkehrsverbund GmbH}(2019)}]{MDVData2019}
\bibinfo{author}{{Mitteldeutscher Verkehrsverbund GmbH}}, \bibinfo{title}{{MDV}
  {GTFS} data},
  \bibinfo{howpublished}{\url{https://www.mdv.de/informationen/downloads/}},
  \bibinfo{year}{2019}. \bibinfo{note}{Accessed: 2019-08-11}.
\bibitem[{Hansen(1980)}]{HansenBicriterion}
\bibinfo{author}{P.~Hansen},
\newblock \bibinfo{title}{Bicriterion path problems},
\newblock \bibinfo{journal}{Lecture Notes in Economics and Mathematical
  Systems} \bibinfo{volume}{177} (\bibinfo{year}{1980}).
  \DOIprefix\doi{10.1007/978-3-642-48782-8_9}.
\bibitem[{{Gabow} et~al.(1976){Gabow}, {Maheshwari}, and {Osterweil}}]{1702370}
\bibinfo{author}{H.~N. {Gabow}}, \bibinfo{author}{S.~N. {Maheshwari}},
  \bibinfo{author}{L.~J. {Osterweil}},
\newblock \bibinfo{title}{On two problems in the generation of program test
  paths},
\newblock \bibinfo{journal}{IEEE Transactions on Software Engineering}
  \bibinfo{volume}{SE-2} (\bibinfo{year}{1976}) \bibinfo{pages}{227--231}.
  \DOIprefix\doi{10.1109/TSE.1976.233819}.
\bibitem[{Kung et~al.(1975)Kung, Luccio, and
  Preparata}]{Kung:1975:FMS:321906.321910}
\bibinfo{author}{H.~T. Kung}, \bibinfo{author}{F.~Luccio},
  \bibinfo{author}{F.~P. Preparata},
\newblock \bibinfo{title}{On finding the maxima of a set of vectors},
\newblock \bibinfo{journal}{J. ACM} \bibinfo{volume}{22} (\bibinfo{year}{1975})
  \bibinfo{pages}{469--476}. \URLprefix
  \url{http://doi.acm.org/10.1145/321906.321910}.
  \DOIprefix\doi{10.1145/321906.321910}.
\bibitem[{Orda and Rom(1990)}]{Orda90shortest-pathand}
\bibinfo{author}{A.~Orda}, \bibinfo{author}{R.~Rom},
\newblock \bibinfo{title}{Shortest-path and minimum-delay algorithms in
  networks with time-dependent edge-length},
\newblock \bibinfo{journal}{J. ACM} \bibinfo{volume}{37} (\bibinfo{year}{1990})
  \bibinfo{pages}{607–625}. \URLprefix
  \url{https://doi.org/10.1145/79147.214078}.
  \DOIprefix\doi{10.1145/79147.214078}.

\end{thebibliography}

\appendix
\section{Computational Complexity and Links to Automata Theory}\label{sec:complexity}
	
In this section, we study the computational complexity of \POEAP.
The intractability of general multi-objective shortest path problems is well-established \citep{HansenBicriterion}.
The standard argument is here, that the output might be exponential in size.
Note that in \POEAP\ the number of tickets $|\tickets|$ is finite.
Therefore, it is always possible to find a valid set of Pareto-optimal solutions with size $\leq |\tickets|$.
This begs the question whether \POEAP\ can be solved in polynomial time. 
The answer depends on whether the monoid $\pmonoid$ is considered as part of the encoding length. 

If $\pmonoid$ is not considered part of the input then even the single-criterion version of \POEAP\
with travel time functions $c\equiv 0$ (denoted by \POEAPN) is NP-hard.
\citet{BlancoBorndoerferHoangetal.2016} proved finding a shortest path with respect to weights from the monoid $(2^Z,\subseteq, \cup)$.
with fare zones $Z$ to be NP-hard.
The proof was obtained using a reduction from the minimum-color single-path problem \citep{broersma2005paths}.  
As this is a special case of \POEAPN, the NP-hardness of \POEAPN{} follows immediately.

In the following, we provide an alternative reduction of the path with forbidden pairs problem to \POEAPN.
Its NP-completeness was established by \citet{1702370}.

\begin{definition}[Path with Forbidden Pairs Problem]
    Let $G=(V,A)$ be a directed graph and $(a_i,b_i), i \in I$ a list of forbidden pairs. 
    Let $s,t\in V$.
    The path with forbidden pairs problem asks whether there is a $s,t$-path  $p$ in $G$ such that 
    $\forall i\in I :  a_i \notin p \vee b_i\notin p.$
    That is, $p$ contains at most one vertex of each pair $(a_i,b_i)$.
\end{definition}
	
\begin{theorem}[NP-hardness of \POEAPN]
    The single-criterion problem \POEAPN{}  is NP-hard. Its canonical decision problem \POEAPND{}  is NP-complete if the evaluation time of $\transfunc$ is polynomially bounded.
\end{theorem}

\begin{proof}
    The canonical decision problem   \POEAPND{} of \POEAPN{} asks whether there is a path $p$ in $G$ with $\price(p) \leq k$ for some $k\in\mathbb{Q}^+$. 
    We show NP-hardness by reducing  the path with forbidden pairs problem to \POEAP.
    Let $G=(V,A)$ be a directed graph, $s,t\in V$ and $ (a_i,b_i), i\in I$, a list of forbidden pairs.
    We construct a  CFN $\farenetwork = \sixtuple$ as follows.
    Let $\ticketgraph =(\tickets,\tarcs)$ be the ticket graph with tickets $\tickets= \{\ticket_1, \ticket_2\}$ and $\tarcs= \{(\ticket_1,\ticket_2)\}$. 
    Furthermore, we define $\pmonoid$ as follows. Let $\monoidset = \{0,1,2\}^I$. 
    The sum of $x,y\in \monoidset$, is defined as $x+y := ( \min(2,x_i+y_i) )_{i\in I}$. We have $x\leq y$ if and only if $x_i \leq y_i$ for all $i\in I$.
    We set the ticket prices to $\price(\ticket_1) = 0$ and $\price(\ticket_2) = 1$.
    Now, we let 
    \begin{align} 
    \wfa(v_1,v_2)&= \begin{cases}  e_i & v_1 = a_i \text{ or } v_1 = b_i \text{ for some }i \in I \\
    0 & \text{otherwise}
    \end{cases} \quad\quad	\Forall  (v_1,v_2) \in A 
    \intertext{and}
        \transfunc( \ticket_1,\monoidel,\fareevent) &= \begin{cases}
    \ticket_2 & \exists i \in I: h_i = 2\\
    \ticket_1 & \text{otherwise}
    \end{cases}
    \end{align}
    where $e_i  \in \{0,1\}^I$ is the standard unit vector with $e_{ii} = 1$. 
    Let $p$  be a simple path.
    W.l.o.g.\ we assume that the last vertex of $p$ is not in $ \bigcup_{i\in I} \{a_i,b_i\}$.
    Now, assume that $p$ has ticket $\tf(\state(p))= \ticket_2$.
    Hence, there must be an $i \in I$ s.t. $\wfs(\state_i) = 2$ and therefore both $a_i\in p $ and $b_i\in p$.
    Conversely, if both $a_i\in p $ and $b_i\in p$,
    we must have $\state_i(p) = 2$ and therefore $\tf(\state(p))= \ticket_2$.
    The CFN $\farenetwork$ can be built in polynomial time, as the weights $\wfa$ can be built in $\mathcal{O}(|A||I|)$ time.
    Hence, the path with forbidden pairs problem can be polynomially reduced to \POEAPN.
   
    It remains to show that \POEAPND{} is in NP.
    When evaluating the fare state of $p$ as many calls to $\transfunc$ have to be performed as there are arcs in $p$.
    Every simple $s,t$-path has clearly $\leq |A|$ edges. 
    Hence, the overall evaluation time of $\transfunc$  for $p$ is polynomially bounded.
    Finally, finding $\price(p)$ from $\state(p)$ requires a simple table-lookup. 
    Hence, checking $\price(p) \leq k$ can be done in polynomial time.	 
\end{proof}

Now, assume that the underlying monoid $\pmonoid$ is of finite size and consider $|\monoidset|$ a part of the encoding length.
Then, POEAP can be solved in polynomial time using techniques already used by \citet{Barrett:2000:FPP:586846.586970} for the regular language constrained shortest path problem.
We begin with a short recapitulation of crucial results from automata \citep{Hopcroft+Ullman/79/Introduction}. 
\begin{definition}[Deterministic finite automaton]
    A \emph{deterministic finite automaton} (DFA) is a 5-tuple $\automaton$, where $Q$ is a finite set of \textit{states}, $\Sigma$ is a finite \textit{input alphabet}, $q_0\in Q$ is the \textit{initial state}, $\mathcal{F}\subseteq Q$ is the set of \textit{final states} and $\delta: Q\times\Sigma \rightarrow Q$ is the \textit{transition function}.
\end{definition}
The words accepted by a DFA are exactly the words of a regular language. Hence, we can use DFAs to define the regular language constrained shortest path problem (REG-ShP)  \citep{Barrett:2000:FPP:586846.586970}.
In this problem, each arc  $a\in A$ of a directed graph $G=(V,A)$ is associated with a letter $\sigma(a) \in \Sigma$. The state of a path $p=(v_0,\dots,v_n)$ is then recursively defined via 
\begin{align*}
    q( (v_0,\dots, v_i) ) &:= \delta(q( (v_0,\dots, v_{i-1})) ,\sigma(v_{i-1},v_i)) \\
q(v_0) &:= q_0.
\end{align*}

\begin{definition}[Regular-language constrained shortest path problem (REG-ShP)]
Let a directed graph $G=(V,A)$, weights $c:A\rightarrow \mathbb{Q}^+$, a DFA  $\automaton$, letters $\sigma(a), a\in A$, a source $s\in V$ and a destination  $t\in V$ be given. 
Find a shortest $s,t$-path $p$  such that $q(p)\in \mathcal{F}$.
\end{definition}
	
Conditional fare networks $\sixtuple$ can be recast as DFAs for a fixed starting stop $v\in V$ and final fare state $\state \in \statespace$ if $|\monoidset|$ is finite.
Thus, POEAP can be tackled by solving a series of formal-language constrained shortest path problems and selecting the price-optimal path from the successful queries. 

Assume the monoid $\pmonoid$ is trivial, i.e. $\monoidset=\{e\}$ with $e+e=e$ for some element $e$. 
Then, the transition functions $\transfunc(\cdot,\cdot,\cdot)$ can be considered independent of $\pmonoid$, and we can construct a DFA $D(s) = \automaton$ in a straightforward manner by setting
$Q:=\tickets$, $\Sigma := \fareevents$, $q_0 := \ticket'$ where $(\ticket',e) = \startstate(s)$ is the initial fare state of $s$ and $\mathcal{F} := \{\ticket\}$ for some $\ticket\in\tickets$. 
The transition function $\delta$ is defined as
\begin{align*}
    \delta(q, \sigma) :=  \transfunc(q,x,\sigma) \quad \Forall q\in Q , \sigma\in\Sigma.
\end{align*}
As there is no direct equivalent for nontrivial $H$ in a DFA, we incorporate it into the state set $Q$.
We set $Q:=\statespace = \tickets \times \monoidset$, $\Sigma := \monoidset \times \fareevents$, $q_0 := \startstate(v)$
and $\mathcal{F} = \{\state\}$ for some fare state $\state\in\statespace$.
The transition function $\delta$ is then defined by
\begin{align*} 
    \delta(q, \sigma) :=  (\transfunc( \tf(q),\wfs(q) + \wfa(\sigma), \ef(\sigma)), \wfs(q) + \wfa(\sigma)) \quad\Forall q=\in Q , \sigma\in\Sigma.
\end{align*}
An example of this transformation can be seen in Figure \ref{fig:transformation}.

\begin{figure}
	
\begin{minipage}{\linewidth}
\centering
\subfloat[Ticket Graph $\ticketgraph$]{
\scalebox{0.7}{
\begin{tikzpicture}
\tikzset{nodestyle/.style={draw,circle,minimum size=1.4cm}}
\node[nodestyle](t1) at (0,0) {\large$\ticket_1$};
\node[nodestyle](t2) at (4,4) {\large$\ticket_2$};
\node[nodestyle](t3) at (8,0) {\large$\ticket_3$};
\tikzset{arcstyle/.style={-latex}}
\tikzset{sl/.style={sloped, anchor=south,auto=false}}
\draw[arcstyle] (t1) to node[sl] {\large$ \mathbbm{1}_{\{\fareevent=\fareevent_1\wedge \monoidel \geq 1\}}$} (t2);
    \draw[arcstyle] (t1) to node[sl] {\large$ \mathbbm{1}_{\{\fareevent=\fareevent_2\wedge \monoidel \geq 1 \}}$} (t3);
    \draw[arcstyle] (t2) to node[sl] {\large$ \mathbbm{1}_{\{\fareevent=\fareevent_2\}}$}(t3);
\end{tikzpicture} } }
%
\subfloat[DFA]{
\scalebox{0.6}{
\begin{tikzpicture}
\tikzset{nodestyle/.style={draw,circle,minimum size=1.4cm}}
\node[nodestyle](t10) at (0,6)  {$(\ticket_1,0)$};
\node[nodestyle](t11) at (0,4)  {$(\ticket_1,1)$};
\node[nodestyle](t12) at (0,2)  {$(\ticket_1,2)$};
\node[nodestyle](t20) at (5,10) {$(\ticket_2,0)$};
\node[nodestyle](t21) at (5,8)  {$(\ticket_2,1)$};
\node[nodestyle](t22) at (5,6)  {$(\ticket_2,2)$};
\node[nodestyle](t30) at (10,6) {$(\ticket_3,0)$};
\node[nodestyle](t31) at (10,4) {$(\ticket_3,1)$};
\node[nodestyle](t32) at (10,2) {$(\ticket_3,2)$};
\tikzset{arcstyle/.style={-latex}}
\tikzset{sl/.style={sloped, anchor=south,auto=false}}
\draw[arcstyle] (t10)  to node[sl] {$(1,\fareevent_1)$}  (t21);
\draw[arcstyle] (t10) to (t22);
\draw[arcstyle] (t10) to (t31);
\draw[arcstyle] (t10) to (t32);
\draw[arcstyle] (t20) to  (t21);
\draw[arcstyle] (t20) to [bend left=45] (t22);
\draw[arcstyle] (t21) to  (t22);
\draw[arcstyle] (t30) to  (t31);
\draw[arcstyle] (t30) to [bend left=45] node[anchor=west] {$(2,\fareevent_1),(2,\fareevent_2)$}  (t32);
\draw[arcstyle] (t31) to  (t32);
\draw[arcstyle] (t20) to (t30);
\draw[arcstyle] (t20) to (t31);
\draw[arcstyle] (t20) to (t32);
\draw[arcstyle] (t21) to (t31);
\draw[arcstyle] (t21) to (t32);
\draw[arcstyle] (t22) to (t32);
\draw[arcstyle] (t11) to (t21);
\draw[arcstyle] (t11) to (t22);
\draw[arcstyle] (t11) to node[sl,xshift=.75cm] {$(0,\fareevent_2)$} (t31);
\draw[arcstyle] (t11) to (t32);
\draw[arcstyle] (t12) to (t22);
\draw[arcstyle] (t12) to node[sl] {$(0,\fareevent_2),(1,\fareevent_2),(2,\fareevent_2)$}(t32);
\end{tikzpicture}	
} }
\medskip
\end{minipage}
\caption{Transformation from CFN to DFA.
The CFN $\farenetwork = \sixtuple$ is given by the ticket graph $\ticketgraph$ depicted in \textbf{(a)}. Possible ticket transitions are given as indicator functions on the arcs. 
The associated fare monoind is $\pmonoid$ with  $\monoidset:=\{0,1,2\}$ and $a+b := \min \{ a + b, 2\}$ and $0\leq1\leq2$ and the fare events are given as $\fareevents=\{\fareevent_1,\fareevent_2\}$. 
The definitions of $\wfa$, $\ef$ and $\startstate$ depend on the routing graph and are omitted in this example.
Our transformation results in the DFA in \textbf{(b)}. There is a state for each element from $\tickets\times\monoidset$.
Each arc depicts a possible state transformation. For each arc, there is at least one letter from $\Sigma = \monoidset\times\fareevents$ that allows this transformation.
There is no arc between $(\ticket_1,0)$ and $(\ticket_2,0)$ as moving from $\ticket_1$ to $\ticket_2$ in the ticket graph requires $\monoidel\geq1$.
There should be a loop at every state, e.g., $(\ticket_1,0)$ transforms into $(\ticket_1,0)$ if letter $(0,\fareevent_1)$ or $(0,\fareevent_2)$ was found. We omit them in \textbf{b)} to not clutter the presentation.
}\label{fig:transformation}
\end{figure}

Note that we can restrict $\Sigma$ to those fare attributes that do really appear on arcs in $\rarcs$.
Hence, we can assume $\mathcal{O}(|\Sigma|) = \mathcal{O}(|\rarcs|)$.
Then, the above DFA can be created in $\mathcal{O}(|Q| + |\Sigma||Q|B) = \mathcal{O}(|\rarcs||\tickets||\monoidset|B)$ assuming the evaluation time of $\transfunc$ to be bounded by a polynomial $B$. 
Note that constructing a DFA $(Q,\Sigma,\delta,q_0,\{f\})$ for all $\state \in \statespace$ is still possible in $\mathcal{O}(|\rarcs||\tickets||\monoidset|B)$ time as the transition function $\delta$ and the state space $Q$ need only be constructed once.

REG-ShP can be solved over such a DFA in $\mathcal{O}(|V||\tickets||\monoidset| \log(|V||\tickets||\monoidset|) + |\rarcs||\tickets||\monoidset|)$ using the algorithm given by \citet{Barrett:2000:FPP:586846.586970}.
Hence, a superset $M$ of the Pareto-set of POEAP can be found in polynomial time by solving $|\statespace| = |\tickets||\monoidset|$ instances of REG-ShP.
As $M$ has at most $|\tickets||\monoidset|$ entries it takes $\mathcal{O}(|\tickets||\monoidset|\log(|\tickets||\monoidset|))$ time to extract the actual Pareto-set \citep{Kung:1975:FMS:321906.321910}.
Thus, we obtain an overall running time of $\mathcal{O}(|V||\tickets|^2|\monoidset|^2 \log(|V||\tickets||\monoidset|) + |\rarcs||\tickets|^2|\monoidset|^2  + |\rarcs||\tickets||\monoidset|B)$.
A slight modification of the proof for Reg-ShP in \citep{Barrett:2000:FPP:586846.586970} allows us to obtain a tighter bound.

\begin{lemma}[POEAP with constant travel time over finite monoids]\label{lemma:finite}
Consider POEAP  over $\farenetwork = \sixtuple$ with $\monoidset$ finite, i.e., $|\monoidset| < \infty$.
Assume all travel time functions are constants and that the evaluation time of  $\transfunc$ can be bounded by a polynomial $B$.
Then, POEAP can be solved in  $\mathcal{O}(|V||\tickets||\monoidset| \log(|V||\tickets||\monoidset|) + |\rarcs||\tickets||\monoidset|B)$.
\end{lemma}

\begin{proof}
Let $D(s) = \automaton$ be the DFA constructed as described above but letting \mbox{$\mathcal{F} := Q = \statespace$}, i.e., now all states
in the automaton are also accepting states.
To keep consistence with automata terminology, we write 
$\sigma(v_1,v_2) := (\wfa(v_1,v_2),\ef(v_1,v_2))$ for $v_1,v_2\in V$.

We construct a product network $G^\times = (V^\times, A^\times)$ of $G$ and $D(s)$ with
\begin{align*}
    V(G^\times) &= V \times Q \\
    E(G^\times) &=  \{(v_1,q_1),(v_2,q_2) | (v_1,v_2)\in E, q_2 = \delta(q_1,\sigma(v_1,v_2))      \}.
\end{align*}
Note that for every $a\in A$ there are at most $|\tickets||\monoidset|$ edges in $E^\times$ and hence $|E^\times| \leq |\rarcs||\tickets||\monoidset|$. Thus, $G^\times$ can be constructed in $\mathcal{O}(|V||\tickets||\monoidset| + |\rarcs||\tickets||\monoidset|)$.
Using Dijkstra's algorithm, we compute a shortest path tree rooted at $(s,q_0)$ in $G^\times$. In particular, we obtain a shortest 
$(s,q_0),(t,q)$-path for all $q \in Q = F = \tickets\times\monoidset$.
Using a Fibonacci heap, Dijkstra's algorithm has a running time of $\mathcal{O}(|V^\times|\log(|V^\times|) + |E^\times| )$.
We can again extract the Pareto-set in $\mathcal{O}(|\tickets||\monoidset|\log(|\tickets||\monoidset|))$ time, giving an overall runtime of  $\mathcal{O}(|V||\tickets||\monoidset| \log(|V||\tickets||\monoidset|) + |\rarcs||\tickets||\monoidset|B)$.
\end{proof}

This result extends naturally to FIFO-travel time functions.

\begin{theorem}[POEAP over finite monoids is polynomial time solvable in $|\monoidset|$]
    Consider POEAP  over the conditional fare network ${\farenetwork = \sixtuple}$ under the  assumption that $|\monoidset| < \infty$,
that the travel time functions $c(a): I \rightarrow I, a\in A$ have the FIFO-property and can be evaluated in constant time,
    and that the evaluation time of $\transfunc$ can be bounded by a polynomial $B$.
    Then, POEAP can be solved in $\mathcal{O}(|V||\tickets||\monoidset| \log(|V||\tickets||\monoidset|) + |\rarcs||\tickets||\monoidset|B)$.
\end{theorem}
\begin{proof}
    It is well-established that time-dependent shortest path problems can be solved using a modified version of Dijkstra's algorithm for FIFO networks \citep{Orda90shortest-pathand}.
    The modified algorithm exhibits the same running time as the standard algorithm if the evaluation time of travel time functions is bounded by a constant.
    Thus, using the modified Dijkstra variant in Lemma \ref{lemma:finite}  solves POEAP with FIFO travel time functions in $\mathcal{O}(|V||\tickets||\monoidset| \log(|V||\tickets||\monoidset|) + |\rarcs||\tickets||\monoidset|B)$.
\end{proof}

\section{Pseudocode for the \MCRAP{} algorithm}\label{app:mcrap}
We provide the pseudocode of the \MCRAP{} algorithm \cite{DBLP:journals/transci/DellingPW15}
adapted to CFNs as developed in Section \ref{sec:fareinrap}.
The function $B.\mathrm{add}((\raptime,\state))$ removes all labels from $B$ that are dominated by $(\raptime,\state)$.

\begin{algorithm}
\DontPrintSemicolon
\small
\caption{Price-optimal \MCRAP{}}\label{algo:mcrap}
\KwData{Time Table $\TimeTable$, CFN $\farenetwork = \sixtuple$, origin $p_s\in\TTP$, destination $p_t\in\TTP$, departure time $\raptime$, number of rounds $K$ }

$B_0(p_s) \leftarrow \{  (\raptime,\startstate(p_s))  \}$\;
mark $p_s$\;
\For{$k=1$ \KwTo $K$}
{
    
    $B_k(\cdot) \leftarrow \emptyset$\;
    $Q \leftarrow \emptyset$ \tcp*{FIFO queue}
    \ForEach(\tcp*[f]{Find first marked stop per route}){marked stop $p$}
    {
        \ForEach{route $r$ with $p \in \TTP(r)$}
        {
            \eIf{$(r,p')\in Q$ for some $p'\in\TTP$}
            {
                \If{$p$ comes before $p'$ in $r$}
                {
                    $Q.\mathrm{remove}(r,p')$\;
                    $Q.\mathrm{add}(r,p)$\;
                }
            }
            {$Q.\mathrm{add}(r,p)$\;}
        }
        unmark $p$\;
    }
        \ForEach(\tcp*[f]{Traverse Routes}){$(r,p)\in Q$}
        {
            $B_r  \leftarrow \emptyset$\;
            \ForEach{stop $p_i$ in $r$ beginning with $p$}
            {
                \ForEach{$(\raptime,\state,d) \in B_r$}
                {
                $\raptime \leftarrow \arr(d,p)${}\tcp*{Update label at new station $p_i$}
                $\wfs(\state) \leftarrow \wfs(\state)+\tripweight(d,p)$\;
                $\tf(\state) \leftarrow \transfunc(\tf(\state),\wfs(\state),\tripevent(d,p))$\;
                \If{$(\raptime,\state)$ not dominated by $\cup_{j\in 1,\dots,k}B_{j}(p)$}
                {
                $B_{k}(p).\mathrm{add}((\raptime,\state))$\tcp*{Found new nondominated label}
                mark $p_i$\;
                }

                }
                \ForEach{ $(\raptime,\state)\in B_{k-1}(p)$ }
                { $\raptime' \leftarrow \raptime + \transtime${}\tcp*{Apply transfer costs}
                    $\wfs(\altstate) \leftarrow \wfs(\state)+\transweight(d,p)$\;
                    $\tf(\altstate) \leftarrow \transfunc(\tf(\state),\wfs(\altstate),\transevent(d,p))$\;
                    \If{$(\raptime',\altstate)$ not dominated by $B_r$}
                    {   
                        $d\leftarrow \min(d: \dep(d,p) \geq \raptime')$ \tcp*{Find next trip}
                        $B_r.\mathrm{add}((\raptime',g,d))$ \tcp*{Add nondominated label to route bag}
                    }
                }
            }
        }

        \ForEach(\tcp*[f]{Process Footpaths}){marked stop $p$}
        {
            \ForEach{footpath $(p,p',l)\in \TTF$}
            {
                \ForEach{$(\raptime, \state) \in B_k(p)$}
                {
                    \If{$(\raptime + l , \state)$ not dominated by $\cup_{j\in 1,\dots,k}B_{j}(p')$}
                    {
                        $B_k(p').\mathrm{add}((\raptime+l,\state))$
                    }
                }
            }
        }
        \BlankLine
        \If(\tcp*[f]{Early Termination Criterion}){no stop is marked} 
        {stop}
    }
\end{algorithm}

\color{black}
\end{document}